\documentclass{siamart250211}
\colorlet{fadedred}{red!30!gray!80!white}
\colorlet{fadedblue}{blue!70!gray!80!white}
\colorlet{fadedgreen}{green!60!blue!80!gray}

\definecolor{beige}{RGB}{245,245,220}

	\hypersetup{
		colorlinks   = true,
		citecolor    = RoyalBlue,
		linkcolor    = RubineRed,
		urlcolor     = Turquoise
		}
\usepackage{algorithm}
\usepackage{algorithmic}
\usepackage{amsmath}   
\usepackage{amssymb}                        
\usepackage{graphicx}                       
\usepackage{subfigure}
\usepackage{mdwlist}												
\usepackage{mathtools}                      
\newcommand{\D}{\mathbb{D}}

\let\classAND\AND
\let\AND\relax
\usepackage{algorithmic}

\let\AND\classAND
\AtBeginEnvironment{algorithmic}{\let\AND\algoAND}

\floatname{algorithm}{Algorithm}

\usepackage{breakcites} 
\usepackage{adjustbox}
\usepackage{circuitikz}
\usetikzlibrary{patterns}

\newsiamthm{claim}{Claim}
\newsiamremark{remark}{Remark}
\newsiamthm{assumption}{Assumption}

\definecolor{Turquoise}{RGB}{64,224,208}
\definecolor{RoyalBlue}{RGB}{65,105,225}
\definecolor{RubineRed}{RGB}{209,0,86}

\newcommand{\K}{{\mathbb{K}^\circ}}
\newcommand{\MR}{\mathbb{R}}
\newcommand{\intset}[2]{\left[#1,#2\right]_{\mathbb{N}}}

\begin{document}

\title{Newton-PIPG: A Fast Hybrid Algorithm for Quadratic Programs in Optimal Control\thanks{Submitted to the editors \today.
}}
\author{
  Dayou Luo\thanks{Department of Applied Mathematics, University of Washington, Seattle, WA 98195 (\email{dayou@uw.edu}). Corresponding author.}
  \and Yue Yu\thanks{Department of Aerospace Engineering and Mechanics, University of Minnesota Twin Cities, Minneapolis, MN 55455 (\email{yuy@umn.edu}).}
  \and Maryam Fazel\thanks{Department of Electrical and Computer Engineering, University of Washington, Seattle, WA 98195 (\email{mfazel@uw.edu}).}
  \and Beh\c{c}et A\c{c}{\i}kme\c{s}e\thanks{William E. Boeing Department of Aeronautics and Astronautics, University of Washington, Seattle, WA 98195 (\email{behcet@uw.edu}).}
}

\headers{Newton-PIPG for QP in Optimal Control}{D. Luo, Y. Yu, M. Fazel, and B. A\c{c}{\i}kme\c{s}e}

\maketitle
\begin{abstract}                          %
We propose Newton-PIPG, an efficient method for solving quadratic programming (QP) problems arising in optimal control, subject to additional set constraints. Newton-PIPG integrates the Proportional-Integral Projected Gradient (PIPG) method with the Newton method, thereby achieving both global convergence and local quadratic convergence. The PIPG method, an operator-splitting algorithm, seeks a fixed point of the PIPG operator. Under mild assumptions, we demonstrate that this operator is locally smooth, enabling the application of the Newton method to solve the corresponding nonlinear fixed-point equation. Furthermore, we prove that the linear system associated with the Newton method is locally nonsingular under strict complementarity conditions. To enhance efficiency, we design a specialized matrix factorization technique that leverages the typical sparsity of optimal control problems in such systems. Numerical experiments demonstrate that Newton-PIPG achieves high accuracy and reduces computation time, particularly when feasibility is easily guaranteed.
\end{abstract}


\section{Introduction}
The development of techniques for solving Quadratic programming (QP) problems is a key enabler for advancing optimal control research, as efficiently solving these problems is often essential and a bottleneck in real-world optimal control applications. For example, Model Predictive Control (MPC), a framework for implementing optimal control, often relies on solving a series of QP problems 
\cite{houska2011auto,accikmecse2011robust, jones2012fast,zometa2013muao}. In addition, Sequential Convex Programming (SCP), a widely used algorithm for nonlinear optimal control, also relies on solving QP problems \cite{neunert2016fast,szmuk2020successive,malyuta2022convex,elango2024successive}. Fast QP solvers are particularly beneficial in these applications, especially in scenarios requiring real-time implementation \cite{malyuta2022convex,kamath2023customized}.

We consider the following QP problem, which includes additional set constraints and frequently arises in optimal control applications:
\newline
\textbf{Problem 1}:
\begin{equation}
\label{equ: original problem}
\begin{array}{ll}
\underset{z}{\text { minimize }} &\sum_{i=0}^{N+1} \sum_{j=1}^{m_i} \frac{1}{2} \rho_{i j}\left\|z_{i j}\right\|^2+q^T  z \\
\text { Subject to } & H_E z+g_E=0, \\
& H_I z+g_I \succeq 0, \\
& z_{i j} \in \D_{i j}=\left\{x \in R^{n_{i j}} \mid \Phi_{i j}(x) \preceq 0\right\},
\end{array}
\end{equation}
 where $z_{ij} \in \mathbb{R}^{n_{z_{ij}}}$, $z_i = \left(z_{i 1}, z_{i 2}, \cdots, z_{i m_i}\right) \in \mathbb{R}^{n_{z_i}}$ , and $z = \left(z_0, z_1, \ldots, z_{N+1}\right) \in \mathbb{R}^{n_z}$. Here, \( z_i \) represents the combination of state and input variables at the \( i \)-th time 
 point. The vector \( z_i \) consists of \( m_i \) components, where each component \( z_{ij} \) has a 
 quadratic weight  \( \rho_{ij} > 0 \),
 The notations \( \preceq \) and \( \succeq \) denote element-wise inequalities for vectors. Vectors $q \in \mathbb{R}^{n_z}$, $g_E \in \mathbb{R}^{n_E}$, and $g_I \in \mathbb{R}^{n_I}$, as well as matrices $H_E \in \mathbb{R}^{n_E \times n_z}$ and $H_I \in \mathbb{R}^{n_I \times n_z}$, are constants. The function $\Phi_{ij}$ is a smooth vector-valued function on $\mathbb{R}^{n_{z_{ij}}}$, defining the convex set $\D_{ij}$ for additional constraints. We will later add requirements on $\Phi_{ij}$ to include commonly used sets, such as balls, second-order cones, boxes, and half-spaces.

Additionally, we require that both $H_E$ and $H_I$ follows the pattern below:

\[\begin{pmatrix}
A_0 & B_0 & 0 & \cdots & \cdots & 0 & 0 \\
0 & A_1 & B_1 & 0 & & \vdots & \vdots \\
0 & 0 & A_2 & B_2 & \ddots & \vdots & \vdots \\
\vdots & & \ddots & \ddots & \ddots & 0 & 0 \\
\vdots & & & \ddots & A_{N-1} & B_{N-1} & 0 \\
0 & \cdots & \cdots & 0 & 0 & A_N & B_N \\
\end{pmatrix} \]
where $A_i$ and $B_i$ are constant matrices. The number of columns of $A_i$ is $n_{z_{i}}$ and that of $B_i$ is $n_{z_{i+1}}$. The number of rows for $A_i$ and $B_i$ are determined by the linear equality $H_E z + g_E = 0$ and the inequality $H_I z + g_I \geq 0$. We denote these numbers as $n_{E_i}$ for the equalities and $n_{I_i}$ for the inequalities in the $i$-th row of the block matrices $H_E$ and $H_I$, respectively.

The sparsity patterns of \( H_E \) and \( H_I \) are typical for optimal control optimization problems. These matrices define equality and inequality constraints among consecutive time points, while the set \( \D_{ij} \) includes constraints for states and inputs at each time point.

There are a few differences between Problem~\hyperref[equ: original problem]{1} and general optimal control QP problems. First, our objective function requires strong convexity, and the quadratic term must be homogeneous with respect to the constraint \( \D_{ij} \). Specifically, for dimensions of \( z \) constrained by the same set \( \D_{ij} \), the corresponding quadratic cost coefficients must be identical. Second, we restrict the set \( \D_{ij} \) to specific families of convex sets. Possible relaxations of these requirements are discussed in Section \ref{sec: application and relaxation}.

Widely used methods for solving Problem~\hyperref[equ: original problem]{1} fall into two primary categories: second-order methods and operator-splitting methods, each with their own advantages and disadvantages. Second-order methods, such as interior-point methods and active-set methods, utilize the curvature information of the optimization problem. These methods typically converge in a modest number of iterations and are robust when the problem approaches infeasibility. Commonly used software in this category includes MOSEK~\cite{andersen2000mosek}, Gurobi~\cite{gurobi}, and ECOS~\cite{domahidi2013ecos}. For further theoretical discussions on second-order methods, we refer the reader to \cite{andersen2000mosek,boyd2004convex,nocedal1999numerical}.

One drawback of second-order methods lies in their high computational cost per iteration, as each step requires solving a linear system. This issue becomes significant as the dimension of the optimization variables increases. Additionally, while interior-point solvers are generally effective, they typically cannot directly handle Problem 1 in its original form and require the use of parsers, such as Yamlip \cite{Lofberg2004} or CVX \cite{cvx}, to transform the problem into a solver-compatible form. This transformation process adds overhead and increases the effort needed for verification, validation, and maintenance.

Another approach to solving Problem~\hyperref[equ: original problem]{1} is the operator-splitting method, which constructs an update operator applied at each iteration, ensuring that the iterations converge to the fixed point of this operator. Unlike interior-point methods, operator-splitting methods eliminate the need for matrix factorization, resulting in significantly lower per-iteration computational costs. Additionally, these methods do not require external parsers and are typically implemented with a smaller codebase \cite{kamath2023customized}. For a comprehensive review of operator-splitting methods, see \cite{jiang2023bregman}. However, one drawback of operator-splitting methods is that they often require a large number of iterations to converge, especially when the problem is ill-conditioned. 




Our main contribution is the development of Newton-PIPG, a hybrid method that integrates the Proportional-Integral Projected Gradient (PIPG) method~\cite{yu2020proportional, yu2022extrapolated}, a representative operator-splitting approach, with the Newton method. This approach is motivated by the insight that finding a fixed point for the updating operator in operator-splitting methods is equivalent to solving a nonlinear fixed point equation. Given the Newton method’s effectiveness in solving such equations, we design a Newton step for this fixed-point problem, with PIPG providing a warm start to ensure global convergence. Compared to second-order methods, Newton-PIPG reduces the need for matrix factorization and avoids the requirement for additional parsers. Additionally, in contrast to operator-splitting methods, the Newton step reduces the number of iterations required, enhancing both accuracy and speed.

We also provide theoretical guarantees for the Newton-PIPG method. Specifically, we demonstrate that under an extended linear independence constraint qualification and strict complementarity condition, the Newton step is well-defined, the associated matrix is nonsingular, and the Newton-PIPG algorithm enjoys global convergence with local quadratic convergence behavior.

Furthermore, we tailor our algorithm for optimal control problems. Recognizing that the most computationally expensive step is solving a large linear system in the Newton method, we exploit the sparsity inherent in optimal control problems to design an efficient factorization strategy, significantly reducing computation time.


\subsection{Related Work}
Several studies have focused on integrating operator splitting methods with Newton-type methods. These approaches include directly using BFGS \cite{dennis1996numerical} to solve the fixed-point equation from the proximal point method \cite{chen1999proximal}, and employing the ADMM method in conjunction with a semi-smooth Newton method \cite{ali2017semismooth}, where linear systems are solved using GMRES \cite{saad1986gmres}. Other approaches involve hybridizing operator-splitting methods with quasi-Newton methods to balance global convergence and local acceleration. Popular choices for quasi-Newton methods include BFGS \cite{themelis2019supermann,sopasakis2019superscs} and Anderson acceleration \cite{zhang2020globally}.

The main difference between our work and others is that we focus on a specific type of optimal control QP problem. This focus provides several advantages for our method. First, we demonstrate that, under mild assumption, the update operator in PIPG is differentiable in a neighborhood of the solution. This suggests that the semi-smooth Newton method essentially acts as a Newton method, offering a quadratic convergence guarantee. Details of this differentiability are presented in Theorem \ref{thm: nonsingularity}.

Second, rather than using quasi-Newton methods or iterative solvers, we directly handle the linear system in the Newton step using an efficient matrix factorization technique. This approach leverages the sparsity pattern of the matrix, resulting in faster computation.

Third, our algorithm offers a stronger theoretical guarantee. We prove, rather than assume as in \cite[Assumption A6]{ali2017semismooth}, that the matrix involved in the Newton step is nonsingular in a local neighborhood of the solution under a generalized Linearly Independent Constraint Qualification (LICQ) assumption, which ensures the robustness of the algorithm. Moreover, our algorithm guarantees convergence even when certain key assumptions are not satisfied.

Beyond these distinctions, our method can also be viewed as an extension of existing Interior Point approaches, particularly the algorithm in \cite{wang2009fast}. The matrix factorization technique used in the Newton step is inspired by the blocked Cholesky decomposition method introduced in \cite{wang2009fast}. Other methods, such as \cite{wright1993interior}, handle similar matrices with different linear algebra techniques but maintain the same computational complexity order for matrix factorization as \cite{wang2009fast}. 
A key distinction between our approach and those in \cite{wang2009fast,wright1993interior} is our use of operator-splitting methods. This strategy not only ensures the applicability of our method to QPs with additional set constraints but also reduces the number of matrix factorizations required for convergence.

\subsection{Road Map} The structure of this paper is as follows. Section~\ref{sec: background} reviews the PIPG method and the Newton method. Section~\ref{sec:Newton} presents the Newton step for the PIPG operator, its theoretical properties, and an efficient factorization method for the matrix involved in the Newton step. Section~\ref{sec:global convergence} discusses the global and local convergence of the Newton-PIPG algorithm. Section~\ref{sec: application and relaxation} outlines the restrictions of our algorithm and possible relaxations. Section~\ref{sec:Numerical implementation} provides implementation details, including the design of termination criteria, and presents numerical experiments comparing our method with other state-of-the-art QP solvers.

\subsection{Notation}
We use the following notation throughout the paper. $\mathbb{R}$ denotes the set of real numbers, with $\mathbb{R}_{+}$ and $\mathbb{R}_{-}$ representing nonnegative and nonpositive real numbers, respectively. The set of $n \times m$ real matrices is $\mathbb{R}^{n \times m}$, and $\mathbb{R}^n$ denotes $n$-dimensional real vectors. The identity matrix is $I$, and the zero vector in $\mathbb{R}^n$ is $0^n$. Vectors are treated as columns, and we write $\left(v_1, v_2, \ldots, v_n\right)$ for a vertically stacked column vector when unambiguous.
For the $i$-th element of a vector ${v}$, we use the notation $v_{(i)} \in \mathbb{R}.$ For the element on the $i$-th row and the $j$-th column of a matrix $M$, we denote it as $M_{(i,j)} \in \mathbb{R}$. We define vector inequalities as element-wise comparisons. For example, for vectors \( a \) and \( b \), \( a \preceq b \) indicates that \( a_{(i)} \leq b_{(i)} \) for all \( i \), and \( a \succeq b \) indicates that \( a_{(i)} \geq b_{(i)} \) for all \( i \). We use $\operatorname{blkdiag}(A_1, A_2, \cdots, A_n)$ to build a block diagonal matrix with $A_1$ to $A_n$ as the diagonal block matrices. $\|\cdot \|$ denotes the Euclidean norm for vectors and matrices. For a convex set $\mathbb{D}$, we denote $\Pi_{\mathbb{D}}(z)$ as the projection of $z$ onto set $\mathbb{D}$ and for $z \in \mathbb{D}$, we denote the normal cone of $\mathbb{D}$ at the point $z$ as $N_{\mathbb{D}}(z)$. For the definition of the normal cone, we refer to \cite[Chapter 2.1]{borwein2006convex}. For general closed set $\D$, we denote the distance of a point $z_0$ to $\D$ as $d_\D = \operatorname{min}\{\|z - z_0\|| z \in \D\},$ and we use $\operatorname{ri}\D$ as the relative interior of $\D.$ We denote $\D_1 \times \D_2$ as the direct product of set $\{\D_1, \D_2\},$ and  $\prod_{i=0}^N \D_i$ as the direct product of set $\{\D_1, \D_2, \cdots, \D_N\}$. We use the notation $o(\epsilon)$ to denote a function $f(\epsilon)$ that $\lim_{\epsilon \to 0} \frac{f(\epsilon)}{\epsilon} = 0$. We use $\intset{1}{n}$ for the integer set $\{1,2,\cdots, n\}.$ We denote \(\operatorname{span}(v)\) as the set of all linear combinations of the vectors in the set \( v \).

\section{Background}
\label{sec: background}

To simplify the notation, we rewrite Problem~\hyperref[equ: original problem]{1} as follows 
\newline
\textbf{Problem 2}:
\label{equ: PIPG problem}
\begin{equation}
\begin{array}{cl}
\underset{z}{\operatorname{minimize}} & \frac{1}{2} z^{\top} P z+q^{\top} z \\
\text { subject to } & H z-g \in \mathbb{K}, \quad z \in \mathbb{D},
\end{array}
\end{equation}
Here, the matrix \(H\) is constructed by rearranging the rows of the matrix $H_E$ and $H_I$, resulting in the following form:
\begin{equation}
\label{equ: H def}
H = \begin{pmatrix}
A_{0} & B_0 & 0 & \cdots & \cdots & 0 & 0 \\
0 & A_1 & B_1 & 0 & & \vdots & \vdots \\
0 & 0 & A_2 & B_2 & \ddots & \vdots & \vdots \\
\vdots & & \ddots & \ddots & \ddots & 0 & 0 \\
\vdots & & & \ddots & A_{N-1} & B_{N-1} & 0 \\
0 & \cdots & \cdots & 0 & 0 & A_N & B_N \\
\end{pmatrix}.    
\end{equation}
In this formulation, the notation \( (A_i, B_i) \) is reused to represent the block matrices that combine the original blocks from \(H_E\) and \(H_I\), corresponding to the equality and inequality constraint coefficients at the \(i\)-th time point. 
Setting \(n_i = n_{E_i} + n_{I_i}\), the matrices \(A_i \in \mathbb{R}^{n_i \times n_{z_i}}\) and \(B_i \in \mathbb{R}^{n_i \times n_{z_{i+1}}}\). The cone \( \mathbb{K} \) is defined as \( \prod_{i=0}^N (0^{n_{E_i}} \times \mathbb{R}_+^{n_{I_i}})\), and the vector \( g \) is from vectors \(g_E \text{ and } g_I\), ensuring that \(H z - g \in \mathbb{K}\) enforces both \(H_E z + g_E = 0\) and \(H_I z + g_I \geq 0\). From now on, \(A_i\) and \(B_i\) refer to the combined blocks in \(H\).

The set \( \mathbb{D} \) is defined as \( \prod_{i = 0}^{N+1} \prod_{j = 0}^{m_i} \D_{ij} \). The dual cone \( \mathbb{K}^\circ \) is \( \prod_{i=0}^N (\mathbb{R}^{n_{E_i}} \times \mathbb{R}_{-}^{n_{I_i}}) \). The dimension of \( \mathbb{K} \) is denoted by \( n_w  = \sum_{i = 0}^Nn_{{E_i}} + n_{{I_i}}\).

\subsection{Proportional-Integral Projected Gradient}
\label{sec:PIPG}
As discussed in the introduction, we propose a new method that combines the Proportional-Integral Projected Gradient (PIPG) method with the Newton method. In this section, we provide a brief overview of the PIPG method.

Developed in \cite{yu2022extrapolated}, the PIPG algorithm's update rule for Problem~\hyperref[equ: PIPG problem]{2} is defined as:
\begin{equation}
\label{PIPG iter}
\begin{aligned}
& z^{k+1}=\pi_{\mathbb{D}}\left[z^k-\alpha\left(P z^k+q+H^{\top} w^k\right)\right] \\
& w^{k+1}=\pi_{\mathbb{K}^{\circ}}\left[w^k+\beta\left(H\left(2 z^{k+1}-z^k\right)-g\right)\right]
\end{aligned}
\end{equation}
where $\pi_\D : \mathbb{R}^{n_z} \rightarrow \mathbb{R}^{n_z}$ is the projection operator on to set $\D,$ and $\pi_{\mathbb{K}^{\circ}}: \mathbb{R}^{n_w} \rightarrow \mathbb{R}^{n_w}$ is the projection operator on to set $\K = \prod_{i=0}^N \mathbb{R}^{n_{E_i}} \times \mathbb{R}_{-}^{n_{I_i}}.$
We denote the operator $T:\mathbb{R}^{n_z}\times \mathbb{R}^{n_w} \rightarrow \mathbb{R}^{n_z + n_w} $ for the PIPG operator and express the PIPG update in \eqref{PIPG iter} as
\begin{equation}
\label{equ:iteration}
(z^{k+1}, w^{k+1}) = T(z^{k}, w^{k}). 
\end{equation} 
Let \( \operatorname{Fix}(T) \) denote the set of fixed points of operator \( T \), that is, \( \operatorname{Fix}(T) = \{(z, w) \in \mathbb{R}^{n_z + n_w} \mid (z, w) = T(z, w)\} \). The following theorem for the convergence of the PIPG method is a direct consequence of Theorem \ref{thm: extended-licq}, Theorem \ref{thm: convergence of PIPG}, and Theorem \ref{thm: km-convergence}, discussed later in Section \ref{sec:global convergence}:
\begin{theorem}
\label{thm: pipg fixed-point}
 Suppose assumptions in Theorem \ref{thm: extended-licq} hold. Let $(z^k,w^k)$ be computed as \eqref{equ:iteration}, \( \alpha, \beta > 0 \) and \( \alpha\|P\| + \alpha\beta\|H\|^2 < 1 \). If \( \ \operatorname{Fix}(T) \) is nonempty, the sequence \( (z^k, w^k) \) will converge to a point \((z^\star, w^\star) \in \operatorname{Fix}(T) \), and $z^\star$ is a solution of Problem~\hyperref[equ: PIPG problem]{2}.
\end{theorem}

For a fixed point \((z^\star, w^\star) \in \operatorname{Fix}(T)\), we denote \(z^\star\) as a primal solution and \(w^\star\) as a dual solution to Problem~\hyperref[equ: PIPG problem]{2}. The pair \((z^\star, w^\star)\) will be referred to as a primal-dual solution pair to Problem~\hyperref[equ: PIPG problem]{2}. Notably, for a point $(z^\star, w^\star)\in \operatorname{Fix}(T)$, we have
\begin{equation*}
\begin{aligned}
& z^{\star} = \Pi_{\mathbb{D}}\left[z^\star - \alpha\left(P z^\star + q + H^{\top} w^\star\right)\right],\label{equ: pipg first} \\
& Hz^{\star} - g\in \mathbb{K}.
\end{aligned}
\end{equation*}
Here the second equation is a consequence of the fact that $N_\K(w^\star) \subset K$

\subsection{Newton Method}
Theorem \ref{thm: pipg fixed-point} indicates that PIPG solves a nonlinear fixed-point equation \((z^\star, w^\star) - T(z^\star, w^\star) = 0\). Consequently, we consider applying the Newton method directly to solve this equation, given its fast local convergence rate~\cite{nocedal1999numerical}. In this section, we provide a brief review of the Newton method.

The Newton method is designed to solve 
\begin{equation*}
R(x)=0,
\end{equation*}
where \(R: \mathbb{R}^n \rightarrow \mathbb{R}^n\) is a continuously differentiable function with a Jacobian matrix $J(x)$. At iteration $x_k$, a linear approximation for $R(x_k)$ is 
$
R(x_k + p) \approx R(x_k) + J(x_k) p.
$ The Newton method chooses $p_k$ such that the linear approximation of $R$ equals zero, i.e., $p_k = -J(x_k)^{-1} R(x_k)$. The next iteration $x_{k+1}$ is defined as 
$
x_{k+1} = x_k + p_k.
$ 

The Newton method is known for its locally quadratic convergence rate, but it also has three major disadvantages: the requirement of a nonsingular Jacobian matrix, the high computational cost of computing a factorization of the Jacobian matrix, and the lack of a global convergence guarantee. Additionally, in order to apply the Newton method to the fixed-point equation of $T$, we need to assess the differentiability of the operator $T$.

In the next section, we will define the Newton step for the PIPG operator $T$. We will show that the Jacobian of the PIPG exists and that the matrix appearing in the Newton step is invertible in the neighborhood of \( T \)'s fixed point under some mild assumptions. Hence, the linear equation in the Newton step has a local solution. Additionally, we will provide a factorization strategy for the Jacobian matrix that leverages its sparse structure. The global convergence result is deferred to Section \ref{sec:global convergence}.


\section{Newton Step for PIPG}
\label{sec:Newton}
We start with a formal definition of the Newton step for PIPG. Consider the case where both $\pi_\D$ and $\pi_\K$ are differentiable. In this case,  the Jacobian of the PIPG operator can be computed using the chain rule:
\begin{equation}
\label{equ: def_J_T}
\begin{aligned}
&J_{T} := \begin{pmatrix}
\frac{\partial z^{k+1}}{\partial z^k} & \frac{\partial z^{k+1}}{\partial w^k} \\
\frac{\partial w^{k+1}}{\partial z^k} & \frac{\partial w^{k+1}}{\partial w^k}
\end{pmatrix}\\
\end{aligned}
\end{equation}
where
\begin{align*}
\frac{\partial z^{k+1}}{\partial z^k} &= J_\D (I - \alpha P), \\
 \frac{\partial z^{k+1}}{\partial w^k} &= J_\D (-\alpha H^\top), \\
\frac{\partial w^{k+1}}{\partial z^k} &= J_\K \beta H (-I + 2 J_\D (I - \alpha P))\\
&=J_\K \beta H (-I + 2 \frac{\partial z^{k+1}}{\partial z^k}), \\
\frac{\partial w^{k+1}}{\partial w^k} &=J_\K + 2 J_\K \beta H J_\D (-\alpha H^\top)\\
&=J_\K + 2 J_\K \beta H \frac{\partial z^{k+1}}{\partial w^k},
\end{align*}
and $J_\D$ and $J_\K$ are the Jacobians of projections $\pi_\D$ and $\pi_\K,$ evaluated at $z^k-\alpha\left(P z^k+q+H^{\top} w^k\right)$  and $w^k+\beta\left(H\left(2 z^{k+1}-z^k\right)-g\right)$ respectively.

If $\pi_\D$ and $\pi_\K$ are not differentiable or if $I - J_T$ is singular, we will simply set $J_T$ as a zero matrix of the same shape. We will later show that regardless of the choice of $J_T$, the global convergence of our algorithm will not be affected and, with mild assumptions, it will also not affect the local convergence rate. Also, note that projections are functions that are almost everywhere differentiable.

With $J_T$ being well-defined for all $(z, w)$, we have the definition of the Newton step

\begin{definition}
The update rule for the Newton step is given by  
\begin{equation*}
    (z^{k+1}, w^{k+1}) = (z^k, w^k) + (\Delta z, \Delta w),
\end{equation*}
where $(\Delta z, \Delta w)$ is the solution of the linear system  
\begin{equation}
\label{equ: newtonstep}
    (I - J_T(z^k, w^k)) (\Delta z, \Delta w) = -(z^k, w^k) + T(z^k, w^k).
\end{equation}
Here, $J_T(z^k, w^k)$ is defined as in \eqref{equ: def_J_T} if $J_\D$ and $J_\K$ exist and $I - J_T$ is nonsingular; otherwise, we set $J_T = 0$.
\end{definition}

\subsection{Extended Linearly Independent Constraint Qualification}

In this section, we introduce the constraint qualification for our problem, which ensures the matrix in the Newton step is locally nonsingular. Such nonsingularity requires the fixed point set of the PIPG operator to be a single point. This, in turn, implies the uniqueness of both the primal and dual solutions: the strong convexity of Problem~\hyperref[equ: original problem]{1} guarantees a unique primal solution, while we need additional constraint qualifications to ensure the uniqueness of the dual solution.

The linearly independent constraint qualification (LICQ) is commonly used to ensure the uniqueness of the dual variable. However, LICQ is trivially violated for some important cases in Problem~\hyperref[equ: original problem]{1}. 

One example where LICQ fails is when \(\D_{ij}\) is a second-order cone and the optimal solution \(z_{ij} = 0\). Note that the second-order cone is defined as 
\[
\D_{ij} = \{(x, y) \in \mathbb{R}^{n_{z_{ij}-1}} \times \mathbb{R} \mid \|x\|^2 - y \leq 0, y \geq 0\},
\]
When \(x = 0\), the gradient of \( \|x\|^2 - y \) is a zero vector, causing LICQ to automatically fail. 

Another example is when \(\D_{ij}\) is an affine subspace. According to our definition of \(\D_{ij}\), an affine subspace is formed as 
\[
\{x \in \mathbb{R}^{n_{z_{ij}}} \mid 0 \preceq Ax - b \preceq 0\}.
\]
For any point \(z_{ij} \in \D_{ij}\), LICQ is automatically violated.

To include both types of set constraints in our discussion, we consider the following extended LICQ:

\begin{definition}[Extended-LICQ]
Problem~\hyperref[equ: original problem]{1} satisfies the extended Linearly Independent Constraint Qualification (extended-LICQ) if
$$\operatorname{range}([H_E^\top, H_{A_I}^\top]) \bigcap \operatorname{span}(N_\mathbb{D}(z^\star)) = \{0\},$$
and the matrix $[H_E^\top, H_{A_I}^\top]$ has full column rank. Here, $\operatorname{range}(\cdot)$ denotes the column space of a matrix. Moreover, $z^\star$ is the optimal solution of Problem~\hyperref[equ: original problem]{1}, and $H_{A_I}$ consists of rows from $H_I$ that activate (achieve equality) at $z^\star$.
\end{definition}

The extended-LICQ, or its equivalent forms, are commonly used constraint qualifications for problems that involve second-order cones. For instance, the extended LICQ is equivalent to the second-order constraint qualification defined in \cite{mordukhovich2014full}.

Now, we need to show that under extended-LICQ, all solutions to Problem~\hyperref[equ: PIPG problem]{2} satisfy the Karush-Kuhn-Tucker (KKT) condition:

\begin{theorem}
\label{thm: extended-licq}
Under the extended-LICQ assumption, $z^\star$ is a solution to Problem~\hyperref[equ: PIPG problem]{2} if and only if $z^\star$ is feasible for Problem~\hyperref[equ: PIPG problem]{2} and there exists a dual variable $w^\star \in \K$ such that
\begin{equation}
\label{equ: kkt}
0 \in Pz^\star + q + H^\top w^\star + N_\mathbb{D}(z^\star),
\end{equation}
and 
\begin{equation*}
H z^\star - g \in N_\K(w^\star).
\end{equation*}
Moreover, if the solution $z^\star$ to Problem~\hyperref[equ: PIPG problem]{2} exists, $(z^\star,w^\star)$ is unique.
\end{theorem}

\begin{proof}
For the if part, by the definition of $N_\K$, $H z^\star - g \in N_\K(w^\star)$ is equivalent to
\begin{equation*}
\begin{aligned}
& w^{\star} \in \mathbb{K}^{\circ}, \\
& H z^{\star}-g \in \mathbb{K}, \\
& w^{\star \top}\left(H z^{\star}-g\right)  =0 .
\end{aligned}
\end{equation*}
Furthermore, \eqref{equ: kkt} ensures that, by \cite[Proposition 2.1.2]{borwein2006convex}, $z^\star$ is a minimizer of the function 
$$
\frac{1}{2} z^{\star\top} P z^\star + q^\top z^\star + w^{\star\top} (H z^\star - g) + \mathbb{I}_\D(z^\star),
$$
where $\mathbb{I}_\D(z^\star)$ is the indicator function of the set $\mathbb{D}$, taking the value $0$ if $z^\star \in \mathbb{D}$ and $+\infty$ otherwise.

Since $w^\star \in \K$, $w^{\star\top} (H z - g) $ is non-positive for all $z$ feasible to Problem~\hyperref[equ: PIPG problem]{2}, and thus
\begin{equation*}
\begin{aligned}
& \frac{1}{2} z^{\star \top} P z^{\star}+q^{\top} z^{\star} \\
& =\frac{1}{2} z^{\star \top} P z^{\star}+q^{\top} z^{\star}+w^{\star \top}\left(H z^{\star}-g\right)+\mathbb{I}_{\mathrm{D}}\left(z^{\star}\right) \\
& \leq \frac{1}{2} z^{\top} P z+q^{\top} z+w^{\star \top}(H z-g)+\mathbb{I}_{\mathbf{D}}(z) \\
& \leq \frac{1}{2} z^{\top} P z+q^{\top} z
\end{aligned}
\end{equation*}
for all feasible $z$. Thus, $z^\star$ is the solution to Problem~\hyperref[equ: PIPG problem]{2}.

For the only if part, \cite[Theorem 10.47]{clarke2013functional} ensures the existence of 
$w^\star \in \K$, $\eta = 0$ or $1$, and $(w^\star, \eta) \neq 0$ such that
\begin{equation}
\label{lo-equ1: kkt}
0 \in \eta (Pz^\star + q) + H^\top w^\star + N_\mathbb{D}(z^\star),
\end{equation}
and 
\begin{equation}
\label{lo-equ: kkt}
w^{\star\top} (H z^\star - g) = 0.
\end{equation}
The extended LICQ condition ensures that \( \eta = 1 \). Suppose instead that \( \eta = 0 \), then \( w^\star \neq 0 \). From \cref{lo-equ: kkt}, we have \( H^\top w^\star \in \operatorname{range}([H_E^\top, H_{A_I}^\top]) \), and the extended LICQ condition ensures that \( [H_E^\top, H_{A_I}^\top] \) has full rank, which implies \( H^\top w^\star \neq 0 \) for any \( w^\star \neq 0 \). However, \eqref{lo-equ1: kkt} also implies that \( H^\top w^\star \in \operatorname{span}(N_\D(z^\star)) \), which contradicts the extended LICQ condition.

The uniqueness of \( z^\star \) follows from the strong convexity of Problem~\hyperref[equ: PIPG problem]{2}. The uniqueness of \( w^\star \) can be shown by contradiction. Suppose there exist \( w^\star_1 \) and \( w^\star_2 \) satisfying~\eqref{equ: kkt}. Then,
\[
H^\top(w^\star_1 - w^\star_2) \in \operatorname{span}(N_\D(z^\star)), \quad
(w^\star_1 - w^\star_2)^\top (H z^\star - g) = 0.
\]
Hence, similar to the proof for the "only if" part, we have \[ 0 \neq H^\top(w^\star_1 - w^\star_2) \in \operatorname{range}([H_E^\top, H_{A_I}^\top]), \]
which contradicts the extended LICQ condition.

\end{proof}

\begin{remark}
In the proof of Theorem \ref{thm: extended-licq}, the extended LICQ is not required for the "if" direction. Meanwhile, for any fixed point \((z^\star, w^\star)\) of the PIPG operator, we have
\begin{itemize}
    \item \(-\left(P z^\star + q + H^\top w^\star\right) \in N_\mathbb{D}(z^\star)\),
    \item \(H(z^\star) - g \in N_\K(w^\star)\),
\end{itemize}
Therefore, Theorem \ref{thm: extended-licq} ensures that every fixed point \((z^\star, w^\star)\) of the PIPG operator is a solution to Problem~\hyperref[equ: PIPG problem]{2}, regardless of whether the extended LICQ holds.
 
\end{remark}

\subsection{Properties of Projections to Convex Sets}
In this section, we discuss the properties of projections onto convex sets. These properties are useful for subsequent proofs. To start, we discuss the eigenvalues of Jacobians of projection functions.

\begin{lemma}
\label{thm: proj_eigenvalue_betwen_0_1}
For a convex set \(\mathbb{D}\) and a point \(z\) outside \(\mathbb{D}\), let \(\pi(z)\) denote the projection of \(z\) onto \(\mathbb{D}\). If the Jacobian of this projection, denoted \(J(z)\), exists and is symmetric, then the eigenvalues of \(J(z)\) are real, non-negative, and do not exceed one.
\end{lemma}

\begin{proof}
Since $J(z)$ is symmetric, all eigenvalues of $J(z)$ are real numbers. 

Suppose that $\lambda$ is an eigenvalue of $J(z)$, and let $v$ be a corresponding eigenvector. We have
\begin{equation*}
\lim _{\varepsilon \rightarrow 0} \frac{\pi(z+\varepsilon v)-\pi(z)}{\varepsilon}=J(z)  v=\lambda v
\end{equation*}
Thus $\pi(z+\varepsilon v)=\pi(z)+\lambda \varepsilon v+o(\varepsilon).$

Since $\pi(z) \in \D,$ the definition of projection ensures that
\begin{equation*}
\|(z+\varepsilon v) - \pi( z+\varepsilon v)\|^2 \leq\|(z+\varepsilon v)-\pi(z)\|^2,
\end{equation*}

Therefore, $$\|\pi(z+\varepsilon v)-\pi(z)+\pi(z)-(z+\varepsilon v)\|^2 \leq\|z+\varepsilon v-\pi (z) \|^2,$$ and hence
$$\|\pi(z+\varepsilon v)-\pi(z)\|^2+2\langle \pi(z+\varepsilon v)-\pi(z), \pi(z)-z-\varepsilon v\rangle \leq 0$$
Thus
\begin{equation*}
\begin{aligned}
& \|\lambda \varepsilon v+o(\varepsilon)\|^2+2\langle\lambda \varepsilon v+o(\varepsilon),-\varepsilon v\rangle \\
& \leq 2\langle\pi(z+\varepsilon v)-\pi(z), z-\pi(z)\rangle
\end{aligned}
\end{equation*}
By the convexity of set $\mathbb{D}$, $\langle \pi(z+\varepsilon v)-\pi(z), z-\pi(z)\rangle \leq 0$(\cite[Proposition 7.4]{clarke2013functional}). 
We have $$\|\lambda \varepsilon v+o(\varepsilon)\|^2+2\langle\lambda \varepsilon v+o(\varepsilon),-\varepsilon v\rangle \leq 0$$ 
Dividing $\varepsilon^2$ from both sides,
\begin{equation*}
\left.\left.\|\lambda v+o(1)\|^2 \leq 2\langle\lambda v+o(1)\right., v\right\rangle 
\end{equation*}
Hence, $\lambda \geq 0$.

On the other hand, since projections are Lipschitz functions with Lipschitz constants equal to one \cite[Exercise 7.5]{clarke2013functional}, all eigenvalues of \(J(z)\) need to be less than or equal to one. Otherwise, we would have
\begin{equation*}
\lim_{\varepsilon \rightarrow 0} \frac{\pi(z + \varepsilon v) - \pi(z)}{\varepsilon} = J(z)  v = \lambda v, \quad \lambda > 1
\end{equation*}
for some \(\lambda\) and \(v\), Then we have 
\(
\|\pi(z + \varepsilon v) - \pi(z)\| > \|\varepsilon v\|
\) for \(\varepsilon\) sufficiently small, violating the Lipschitz property of projections.\end{proof}

Lemma \ref{thm: proj_eigenvalue_betwen_0_1} requires \(J(z)\) to be symmetric. This assumption holds for many types of convex sets. For example, projections onto points, boxes, balls, second-order cones, and half-spaces have symmetric Jacobians. These sets cover many applications in optimal control. 

Next, we explore the smoothness of projections onto \(\D_{ij}\). Based on Corollary $4.13$ in \cite{lewis2002active} and Theorem $3.3$ in \cite{hare2004identifying}, we have the following theorem. 
\begin{theorem}
\label{projection theorem}
Consider a point $z_0$ in a convex set 
$$
\mathbb{E} = \left \{z\in \MR^{n_\mathbb{E}}: \Phi_i(z) \leq 0,  \text{for all }i \in \intset{1}{m_\mathbb{E}}  \right\},
$$
where  $\Phi_i: \mathbb{R}^{n_\mathbb{E}} \rightarrow \mathbb{R}$ are smooth convex functions. Denote $I_\mathbb{E}(z_0) = \left\{k: \Phi_k(z_0) = 0 \right\}.$ Assume that $\left\{\nabla \Phi_k\left(z_0\right): k \in I_\mathbb{E}(z_0) \right\}$ is a linearly independent set or an empty set. For any normal vector $\bar n$ in $ri N_\mathbb{E}(z_0)$, there exists a sufficiently small open neighborhood $K$ of $z_0 + \bar n$ such that the projection mapping $\pi_\mathbb{E}$ from $K$ to set $\mathbb{E}$ is a smooth function. 

Furthermore, we set
\begin{equation*}
\mathbb{M} = 
\begin{cases}
\{z : \Phi_k(z) = 0, \ \forall k \in I_\mathbb{E}(z_0)\}, & \text{if } I_\mathbb{E}(z_0) \neq \emptyset, \\
\mathbb{R}^{n_\mathbb{E}}, & \text{otherwise}.
\end{cases}
\end{equation*}
Let \(\pi_\mathbb{M}\) be the projection function onto \(\mathbb{M}\). It follows that
\[
\pi_\mathbb{M} = \pi_\mathbb{E},
\]
for all points in \(K\).

\end{theorem}

Based on Theorem \ref{projection theorem}, we have:

\begin{theorem}
\label{thm: nullspace of jacobian}
Consider the set \(\mathbb{E}\), point \(z_0\), vector \(\bar{n}\), and neighborhood \(K\) in Theorem \ref{projection theorem}. Under assumptions in Theorem \ref{projection theorem}, the projection 
$\pi_\mathbb{E}(y)$ is a smooth function for $y \in K$. The null space of the Jacobian of $\pi_\mathbb{E}(z_0 +\bar n)$ is a subset of \(\operatorname{span} N_{\mathbb{E}}(z_0)\),  the set of all linear combinations of the vectors in the set $N_{\mathbb{E}}(z_0)$.
\end{theorem}

\begin{proof}
Set \(y_0 = z_0 + \bar{n}\). If \(I_\mathbb{E}(z_0) := \{k \mid \Phi_k(z_0) = 0\}\) is empty, then \(z_0 \in \operatorname{int} \mathbb{E}\) and thus \(\bar{n} = 0\). Consequently, the Jacobian of the projection is the identity matrix, and the conclusion holds automatically.

When \(I_\mathbb{E}(z_0)\) is nonempty, without loss of generality, we assume that \(I_\mathbb{E}(z_0) = \{1, \ldots, m\}\), i.e., only the first \(m\) inequalities defining \(\mathbb{E}\) are satisfied as equalities at \(z_0\).

For any \(y \in K\), Theorem \ref{projection theorem} ensures that \(\pi_\mathbb{E}(y) = \pi_\mathbb{M}(y)\). Since \(\pi_\mathbb{M}(y) = \operatorname{argmin}\{\frac 1 2 \|\bar{y} - y\|^2 \mid \bar{y} \in \mathbb{M}\}\), the linear independence condition in Theorem \ref{projection theorem} ensures the following KKT condition: there exists a dual variable \(\lambda(y) = (\lambda_1, \cdots, \lambda_m) \in \mathbb{R}^m\) such that
\begin{equation}
\label{equ: proj_kkt}
0 = (\pi_\mathbb{E}(y) - y) + \sum_{i=1}^m \lambda_i(y) \nabla \Phi_i(\pi_\mathbb{E}(y)).
\end{equation}

Denote the matrix $F(y)$ as 
\[
 \left[\nabla \Phi_1(\pi_\mathbb{E}(y)), \quad \nabla \Phi_2(\pi_\mathbb{E}(y)), \quad \ldots, \quad \nabla \Phi_m(\pi_\mathbb{E}(y))\right].
\]
According to the assumptions in Theorem \ref{projection theorem}, the columns of \(F(y_0)\) are linearly independent. Thus, \(F(y_0)\) has full column rank, and the number of rows in \(F(y_0)\) is not less than the number of columns. Therefore, we can choose \(m\) rows of \(F(y)\), denoted as \(k_1, k_2, \dots, k_m\), to form a matrix \(\tilde{F}(y)\) such that \(\tilde{F}(y_0)\) is invertible. 

As \(\pi_\mathbb{E}\) is a smooth function locally, there exists an open neighborhood \(\tilde{K} \subseteq K\) of \(y_0\) such that for any \(y \in \tilde{K}\), \(\tilde{F}(y)\) is invertible. Therefore, \eqref{equ: proj_kkt} implies that
\[
\lambda(y) = - \tilde{F}^{-1}(y)  [\pi_\mathbb{E}(y)_{k_1} - y_{k_1}, \cdots, \pi_\mathbb{E}(y)_{k_m} - y_{k_m}]^\top.
\]

Using Cramer's rule, \(\tilde{F}^{-1}(y)\), and thus \(\lambda(y)\), are smooth functions for \(y \in \tilde{K}\). We then compute the Jacobian of \eqref{equ: proj_kkt} evaluated at the point \(y_0\), leading to 
\begin{equation}
\label{loc_equ:gradient of jacob}
J_{\mathbb{E}} - I + \sum_{i=1}^m \lambda_i \nabla^2 \Phi_i  J_{\mathbb{E}} + \sum_{i=1}^m \nabla \Phi_i  \nabla \lambda_i = 0,
\end{equation}
where \(J_\mathbb{E}(y)\) is the Jacobian of $\pi_\mathbb{E}(y).$ In \eqref{loc_equ:gradient of jacob}, \(J_\mathbb{E}, \ \lambda_i\) and \(\nabla \lambda_i\) are evaluated at \(y_0\), and \(\nabla \Phi_i\) and \(\nabla^2 \Phi_i\) are evaluated at \(z_0\).

Let \(v\) be a vector in the null space of \(J_{\mathbb{E}}(y_0)\), i.e. \(J_{\mathbb{E}}(y_0) v = 0\). Multiplying both sides of \eqref{loc_equ:gradient of jacob} by \(v\), we obtain 
\[
\sum_{i=1}^m \nabla \Phi_i(y_0)  \nabla \lambda_i^\top(y_0) v = v.
\]

Hence,
\[
v \in \operatorname{span}(\{\nabla \Phi_k(z_0) \mid k \in I_\mathbb{E}(z_0)\}) = \operatorname{span}(N_{\mathbb{E}}(z_0)).
\]
The last equality follows from the fact that \(N_{\mathbb{E}}(z_0) = \{\sum_k \eta_k \nabla \Phi_k(z_0) \mid \eta_k \geq 0,\  k \in I_\mathbb{E}(z_0)\}\)\cite[Thm. 10.39, Cor. 10.44, Thm. 10.45]{clarke2013functional}.\end{proof}

Although Theorem \ref{projection theorem} covers most sets of interest, it excludes projections onto second-order cones and affine subspaces, as both violate its assumptions. To incorporate these cases into our algorithm, we must ensure the smoothness guaranteed by Theorem \ref{projection theorem} also holds. Projections onto affine subspaces are linear and explicitly computable via projection matrices, ensuring smoothness and the validity of Theorem \ref{thm: nullspace of jacobian}. For second-order cones, smoothness is only an issue when the projection lands at the origin, leading to the following lemma.


\begin{lemma}
\label{lemma: second_order_cone null space}
Consider a second-order cone 
\[
\mathbb{E} = \{(x, y) \in \mathbb{R}^{n_x} \times \mathbb{R} \mid \|x\|^2 - ty \leq 0, y \geq 0\},
\]
where \(t\in \MR_+\) is fixed. Then
\[
N_\mathbb{E}(0) = \{(x, y) \in \mathbb{R}^{n_x} \times \mathbb{R} \mid t\|x\|^2 \leq -y, y \leq 0\},
\]
and \(\operatorname{span} N_\mathbb{E}(0) = \mathbb{R}^{n_x + 1}\). For any \(\bar{n} \in \operatorname{ri} N_\mathbb{E}(0)\), there exists an open neighborhood \(K\) around \(\bar{n}\) such that \(\pi_\mathbb{E}(v) = 0\) for all \(v \in K\), ensuring the smoothness of \(\pi_\mathbb{E}\) in \(K\).
\end{lemma}

\begin{proof}
The form of the normal cones can be verified using the definition of normal cones to convex sets. Clearly, \(\operatorname{span}(N_\mathbb{E}(0))\) encompasses the entire space and \(\operatorname{ri} N_\mathbb{E}(0) = \operatorname{int}(N_\mathbb{E}(0))\). Moreover, if \(\bar{n} \in \operatorname{int}(N_\mathbb{E}(0))\), there exists a neighborhood \(K\) such that for any \(v\) in \(K\), \(v \in N_\mathbb{E}(0)\) and \(\pi_\mathbb{E}(v) = 0\), according to the properties of the normal cone.\end{proof}

\subsection{Nonsingularity of the Newton Step Matrix}
In this section, we will show that the Newton step is well defined in a local neighborhood of \((z^\star, w^\star)\), i.e., $T$ is differentiable and \(I - J_T(z^k, w^k)\) is invertible.

We start with the following lemma, which is also a major component of the algorithm. This lemma ensures that the linear system \eqref{equ: newtonstep}, which involves a nonsymmetric matrix, can be transformed into a symmetric linear system. We will later show that this symmetric linear system exhibits a desirable sparsity pattern.

\begin{lemma}
\label{lemma: simplification}
Assume that \( J_\mathbb{D} \) in \eqref{equ: def_J_T} is symmetric and has an eigenvalue decomposition \( J_{\mathbb{D}} = Q_\mathbb{D} \Lambda_\mathbb{D} Q_\mathbb{D}^\top \). Then, the matrix \( I - \Lambda_{\mathbb{D}} + \alpha \Lambda_\mathbb{D} P \) is invertible.
Denote
\begin{equation*}
\begin{aligned}
V_{\mathbb{D}}&:=Q_{\mathbb{D}}\left(I-\Lambda_{\mathbb{D}}+\alpha \Lambda_{\mathbb{D}} P\right)^{-1} Q_{\mathbb{D}}^{\top} \\
W_{\mathbb{D}}&:=H Q_{\mathbb{D}}\left(I-\Lambda_{\mathbb{D}}+\alpha \Lambda_{\mathbb{D}} P\right)^{-1} \Lambda_{\mathbb{D}} Q_{\mathbb{D}}^{\top} H^{\top} \\
\tilde{W}_{\mathbb{D}}&:=\alpha \beta J_{\mathbb{K}^{\circ}} W_{\mathbb{D}} J_{\mathbb{K}^{\circ}}+I-J_{\mathbb{K}^{\circ}} .
\end{aligned}
\end{equation*}
Solving \eqref{equ: newtonstep} is thus equivalent to solving
\begin{equation*}
\begin{aligned}
\tilde{W}_\mathbb{D} \Delta w &= (I - \alpha \beta^2 J_{\mathbb{K}^{\circ}} W_\mathbb{D} (I - J_{\mathbb{K}^{\circ}})) \bar{R}_w^k, \\
\Delta z &= V_\mathbb{D}(\tilde{R}_z^k - \alpha J_\mathbb{D} H^\top \Delta w).
\end{aligned}
\end{equation*}

Here,
\[
\left(\begin{array}{c}
\tilde{R}_z^k \\
\tilde{R}_w^k
\end{array}\right) := \left[\begin{array}{cc}
I & 0 \\
-2J_\K \beta H & I
\end{array}\right]  (T(z^k, w^k) - (z^k, w^k)),
\]
and 
\[
\bar{R}_w^k = J_{\K}\beta H V_\mathbb{D} \tilde{R}_z^k + \tilde{R}_w^k.
\]
The coefficients $\alpha$ and $\beta$ are from the PIPG operator.
\end{lemma}

\begin{proof}
The invertibility of \(I - \Lambda_\mathbb{D} + \alpha \Lambda_\mathbb{D} P\) is a direct consequence of Lemma \ref{thm: proj_eigenvalue_betwen_0_1} and the fact that \(P\) is a diagonal matrix with positive diagonal elements.

We start by simplifying \eqref{equ: newtonstep} to show the second result. Using the definition of $J_T$ in \eqref{equ: def_J_T}, one can verify that 
\[
I - J_T = \left[\begin{array}{cc}
I & 0 \\
2 J_\K \beta H & I
\end{array}\right]  \left[\begin{array}{cc}
I - J_\mathbb{D}(I - \alpha P) & \alpha J_\mathbb{D} H^\top \\
-J_\K \beta H & I - J_\K
\end{array}\right].
\]
Multiplying $$
\left[\begin{array}{cc}
I & 0 \\
2 J_\K \beta H & I
\end{array}\right]^{-1} = \left[\begin{array}{cc}
I & 0 \\
-2 J_\K \beta H & I
\end{array}\right]
$$
on both sides of \eqref{equ: newtonstep}, we obtain the following equivalent Newton step:
\begin{equation}
\label{equ:newton simplification}
\begin{aligned}
&\left[\begin{array}{cc}
I & 0 \\
-2 J_\K \beta H & I
\end{array}\right](I - J_T(z^k, w^k))\left(\begin{array}{c}
\Delta z \\
\Delta w
\end{array}\right) \\
&= \left[\begin{array}{cc}
I - J_\mathbb{D}(I - \alpha P) & \alpha J_\mathbb{D} H^\top \\
-J_\K \beta H & I - J_\K
\end{array}\right] \left(\begin{array}{c}
\Delta z \\
\Delta w
\end{array}\right) \\
&= \left[\begin{array}{cc}
I & 0 \\
-2 J_\K \beta H & I
\end{array}\right]  \left(T(z^k, w^k) - (z^k, w^k)\right) \\
&=: \left(\begin{array}{c}
\tilde{R}_z^k \\
\tilde{R}_w^k
\end{array}\right),
\end{aligned}
\end{equation}
where \(\tilde{R}_z^k \in \mathbb{R}^{n_z}\) and \(\tilde{R}_w^k \in \mathbb{R}^{n_w}\).

Hence, solving \ref{equ: newtonstep} is equivalent to solving the following linear system:
\begin{equation}
\label{local-equ: simplification newton step}
\left[\begin{array}{cc}
I - J_\mathbb{D}(I - \alpha P) & \alpha J_\mathbb{D} H^\top \\
-J_\K \beta H & I - J_\K
\end{array}\right] \left(\begin{array}{c}
\Delta z \\
\Delta w
\end{array}\right) = \left(\begin{array}{c}
\tilde{R}_z^k \\
\tilde{R}_w^k
\end{array}\right).
\end{equation}

We solve this linear system by eliminating \(\Delta z\) using the first row and obtaining an equation involving only \(\Delta w\). To do so, we need the inverse of \(I - J_\mathbb{D}(I - \alpha P)\). Noticing that the specific structure in Problem~\hyperref[equ: original problem]{1} ensures that \(Q_\mathbb{D} P = P Q_\mathbb{D}\), we have 
\begin{equation*}
\begin{aligned}
\left( I - J_\mathbb{D}(I - \alpha P) \right)^{-1} &= \left(I - Q_\mathbb{D} \Lambda_\mathbb{D} Q_\mathbb{D}^\top ( I - \alpha P)\right)^{-1} \\
&= \left(Q_\mathbb{D}\left(I - \Lambda_\mathbb{D} + \alpha \Lambda_\mathbb{D} P\right) Q_\mathbb{D}^\top\right)^{-1} \\
&= Q_\mathbb{D}\left(I - \Lambda_\mathbb{D} + \alpha \Lambda_\mathbb{D} P\right)^{-1} Q_\mathbb{D}^\top \\
&=: V_\mathbb{D}.
\end{aligned}
\end{equation*}
Meanwhile \begin{equation}
\label{loc_equ:vJ}
V_{\mathbb{D}} J_{\mathbb{D}}=Q_{\mathbb{D}}\left(I-\Lambda_{\mathbb{D}}+\alpha \Lambda_{\mathbb{D}} P\right)^{-1} \Lambda_{\mathbb{D}} Q_{\mathbb{D}}^{\top} .
\end{equation}
Solving the first row of \eqref{local-equ: simplification newton step}, we have
\[
\Delta z = V_\mathbb{D}\left(\tilde{R}_z^k - \alpha J_\mathbb{D} H^\top \Delta w\right).
\]

Substituting this into the second row of \eqref{local-equ: simplification newton step}, using \eqref{loc_equ:vJ}, and applying the definitions of \(W_\mathbb{D}\) and \(\bar{R}_w^k\), we obtain
\[
J_\K \alpha \beta W_\mathbb{D} \Delta w + (I - J_\K) \Delta w = \beta \bar{R}_w^k.
\]

Because of the specific form of \(\K\), \(J_\K\) is a diagonal matrix with only zeros and ones on the diagonal. Therefore, \((I - J_\K)J_\K = 0\), and we have  
\[
\begin{aligned}
(I - J_\K) \Delta w &= (I - J_\K) \beta \bar{R}_w^k, \\
\alpha \beta J_\K W_\mathbb{D} J_\K \Delta w &= J_\K \beta \bar{R}_w^k - \alpha \beta J_\K W_\mathbb{D} (I - J_\K) \Delta w.
\end{aligned}
\]

Summing both equations, we have an equivalent linear equation
\[
\begin{aligned}
&(\alpha \beta J_{\mathbb{K}^{\circ}} W_\mathbb{D} J_{\mathbb{K}^{\circ}} + I - J_{\mathbb{K}^{\circ}}) \Delta w \\
&= \beta \bar{R}_w^k - \alpha \beta J_{\mathbb{K}^{\circ}} W_\mathbb{D} (I - J_{\mathbb{K}^{\circ}}) \Delta w \\
&= (\beta I - \alpha \beta^2 J_{\mathbb{K}^{\circ}} W_\mathbb{D} (I - J_{\mathbb{K}^{\circ}})) \bar{R}_w^k.
\end{aligned}
\]\end{proof}

Lemma \ref{lemma: simplification} indicates that the nonsingularity of \(I - J_T\) is equivalent to the nonsingularity of $\tilde{W}_{\mathbb{D}}$. To demonstrate the nonsingularity of the latter, we need to make a few assumptions.

\begin{assumption}
\label{assumption licq}

The solution \(z^\star\) of Problem~\hyperref[equ: PIPG problem]{2} satisfies the extended LICQ condition. Moreover, in Problem~\hyperref[equ: original problem]{1}, each set \(\D_{ij}\) is either a second-order cone, an affine subspace, or satisfies the linear independence condition in Theorem \ref{projection theorem} at all points.

\end{assumption}

\begin{assumption}
\label{assum: symmetric Jacobian}
The Jacobian of projections onto set \(\mathbb{D}\), when it exists, is a symmetric matrix.
\end{assumption}

Since \(\mathbb{D}\) is the direct product of sets \(\D_{ij}\), Assumption \ref{assum: symmetric Jacobian} is satisfied if and only if the Jacobians of projections onto \(\D_{ij}\) are symmetric. Many sets of interest in optimal control have this property, including balls, second-order cones, boxes, points, half-spaces, hyperplanes, affine spaces, and polytopes. This ensures the applicability of our algorithm to a wide range of problems.

To ensure the differentiability of the PIPG operator, we need the following assumption:

\begin{assumption}[Strict Complementarity]
\label{assumption: interior-point}
For the fixed point \((z^\star, w^\star)\) of the PIPG operator, we require
\begin{itemize}
    \item \(-\left(P z^\star + q + H^\top w^\star\right) \in \operatorname{ri} N_\mathbb{D} (z^\star)\)
    \item \(H(z^\star) - g \in \operatorname{ri} N_\K(w^\star)\)
\end{itemize}
\end{assumption}

Assumption \ref{assumption: interior-point} is closely related to the conventional strict complementarity condition \cite[Example 4.4]{hare2004identifying}, which is usually required for a fast convergence rate in Newton-type methods \cite[Assumption 19.1]{nocedal1999numerical}.

Additionally, Assumption \ref{assumption: interior-point} is generally satisfied in practice. When \((z^\star, w^\star)\) are fixed and the vectors \(q\) and \(g\) are sampled from a Gaussian distribution, Assumption \ref{assumption: interior-point} is violated only if \(q\) lies on the relative boundary of the set \(N_\mathbb{D}(z^\star) + Pz^\star + H^\top w^\star\), or if \(-g\) lies on the relative boundary of \(N_\K(w^\star) - Hz^\star\). The conditional probability of violating Assumption \ref{assumption: interior-point}  for fixed \((z^\star, w^\star)\) is zero since the relative boundary has measure zero under the distributions of \(q\) and \(g\). Furthermore, integrating over all \((z^\star, w^\star)\) outcomes, the probability of violation remains zero. Of course, \(q\) and \(g\) may be generated by other types of mechanisms, and it is possible that Assumption \ref{assumption: interior-point} may be violated. In that case, our algorithm will still ensure convergence, though the theoretical convergence rate may not be quadratic.

We are now ready to introduce the theorem for the nonsingularity of $I - J_T(z,w).$
\begin{theorem}
\label{thm: nonsingularity}
Under Assumptions \ref{assumption licq}, \ref{assum: symmetric Jacobian}, and \ref{assumption: interior-point}, \(I - J_T(z,w)\) is a smooth function of \((z, w)\) in a local neighborhood of \((z^\star, w^\star)\), the fixed point of the PIPG operator. Moreover, \(I - J_T(z,w)\) is smooth for \((z,w)\) in that neighborhood.
\end{theorem}

\begin{proof}
We start with the differentiability of \(I - J_T(z, w)\). Utilizing the following fact \cite[Section 6]{rockafellar1970convex},
\[
\operatorname{ri} N_\mathbb{D}(z) = \prod_{i=1}^N \prod_{j=1}^{m_i} \operatorname{ri} N_{\mathbb{D}_{ij}}(z_{ij}),
\]
and according to Theorem \ref{projection theorem}, Lemma \ref{lemma: second_order_cone null space}, Assumption \ref{assumption licq} and Assumption \ref{assumption: interior-point}, it is ensured that \(\pi_\mathbb{D}\) is a smooth function in a neighborhood of \(z^\star - \alpha(Pz^\star + q + H^\top w^\star)\). Similarly, \(\pi_\K\) is a smooth function in a neighborhood of \(w^\star + \beta (H(z^\star) - g)\). Hence, \(I - J_T(z, w)\) is also a smooth function in a neighborhood of \((z^\star, w^\star)\).

Next, we move to the invertibility. According Lemma \ref{lemma: simplification}, \(I - J_T(z^\star, w^\star)\) is invertible if and only if 
\[
\tilde{W}_{\mathbb{D}} =\alpha \beta J_\K H \left(Q_\mathbb{D} (I - \Lambda_\mathbb{D} + \alpha \Lambda_\mathbb{D} P)^{-1} \Lambda_\mathbb{D} Q_\mathbb{D}^\top\right) H^\top J_\K + I - J_\K
\]
is invertible. Let \(v\) be a vector such that \(\tilde{W}_{\mathbb{D}}v = 0\). Since $J_\K$ is a diagonal matrix with zeros and ones on the diagonal, and 
$$
 H \left(Q_\mathbb{D} (I - \Lambda_\mathbb{D} + \alpha \Lambda_\mathbb{D} P)^{-1} \Lambda_\mathbb{D} Q_\mathbb{D}^\top\right) H^\top
$$
is positive semi-definite, we have 
\begin{equation}
\label{local_equ: null space}
\begin{aligned}
(I - J_\K) v &= 0, \\
\Lambda_\mathbb{D} Q_\mathbb{D}^\top H^\top J_\K v &= 0.
\end{aligned}
\end{equation}

Observe that the null space of \(\Lambda_\mathbb{D} Q_\mathbb{D}\) coincides with the null space of \(J_\mathbb{D}\). Additionally, \(J_\mathbb{D}\) is a block diagonal matrix. The null space for each diagonal block is a subset of \(\operatorname{span}(N_{\mathbb{D}_{ij}}(z^\star_{ij}))\), as ensured by Assumption \ref{assumption licq}, Theorem \ref{thm: nullspace of jacobian} and Lemma \ref{lemma: second_order_cone null space}. Hence, the null space of \(J_\mathbb{D}\) is a subset of 
\[
\operatorname{span}(N_\mathbb{D}(z)) = \prod_{i=1}^N \prod_{j=1}^{m_i} \operatorname{span}(N_{\mathbb{D}_{ij}}(z_{ij})),
\]
and \(H^\top J_\K v \in \operatorname{span}(N_\mathbb{D}(z))\). Meanwhile, \(H^\top J_\K\) shares the same column space with \([H_E^\top, H_{A_I}^\top]\), where \(H_E\) and \(H_{A_I}\) are defined in the extended LICQ assumption. Hence, the extended LICQ ensures that \(J_\K v = 0\). Since \((I - J_\K) v = 0\), it follows that \(v = 0\) and thus \(I - J_T\) is invertible at \((z^\star, w^\star)\). As \(I - J_T\) is a smooth function of \((z, w)\) in a neighborhood of \((z^\star, w^\star)\), it remains invertible in some neighborhood of \((z^\star, w^\star)\) .\end{proof}
\begin{remark}
\(I - J_T\) may not be invertible if the current iteration is not close to the primal-dual solution pair. For methods to handle the singular \(I - J_T\), see Section \ref{sec:Numerical implementation}. Importantly, this non-invertibility does not compromise the global convergence or the local convergence rate.

\end{remark}

\subsection{Efficient Computation of Linear System in Newton Step}
\label{sec: chol}

To efficiently factorize the invertible matrix 
$$\tilde{W}_\D := \alpha \beta J_\K W_{\mathbb{D}} J_\K + I - J_\K,$$ 
we use an approach from \cite{wang2009fast}. This approach leverages the sparsity inherent in the optimal control problem, reducing computational costs to a linear function of the total number of time points \(N\).

By definition, 
$$W_\D = HQ_\D(I - \Lambda_{\mathbb{D}} + \Lambda_{\mathbb{D}} \alpha P)^{-1} \Lambda_{\mathbb{D}} Q_\D^{\top} H^{\top}.$$ 
Since $Q_\D(I - \Lambda_{\mathbb{D}} + \Lambda_{\mathbb{D}} \alpha P)^{-1} \Lambda_{\mathbb{D}} Q_\D$ is a block diagonal matrix, we can write it as:
\begin{equation*}
\begin{aligned}
Q_\D(I - \Lambda_{\mathbb{D}} + \Lambda_{\mathbb{D}} \alpha P)^{-1} \Lambda_{\mathbb{D}} Q_\D = 
 \operatorname{blkdiag}(U_0, U_1, U_2, \ldots, U_N, U_{N+1}),
\end{aligned}
\end{equation*}
where $U_i \in \MR^{n_{z_{i}} \times n_{z_i}}$. Hence, $W_\D$ will have the following pattern:

\begin{equation}
\label{equ:compute V}
\begin{aligned}
W_D
    &=\left[\begin{array}{cccccc}
W_{00} & W_{01} & 0 & \cdots & 0 & 0 \\
W_{10} & W_{11} & W_{12} & \cdots & 0 & 0 \\
0 & W_{21} & W_{22} & \cdots & 0 & 0 \\
\vdots & \vdots & \vdots & \ddots & \vdots & \vdots \\
0 & 0 & 0 & \cdots & W_{N-1, N-1} & W_{N-1, N} \\
0 & 0 & 0 & \cdots & W_{N, N-1} & W_{N, N}
\end{array}\right],
\end{aligned}
\end{equation}
where $W_{i,i}\in \mathbb{R}^{n_i\times n_i},$ $W_{i, i+1}\in \mathbb{R}^{n_i\times n_{i+1}}$, $W_{i+1, i}\in \mathbb{R}^{n_{i+1}\times n_i}$ given that $n_i = n_{E_i} + n_{I_i}.$ Using the definition of $H$ in \eqref{equ: H def}, we have
\begin{equation}
\label{equ:newton equ3}
\begin{aligned}
W_{i, i}= &A_{i}U_{i} A_{i}^\top+B_i U_{i+1} B_i^\top,\  i\in [0, N]_\mathbb{N}, \\
W_{i, i+1}= & W_{i+1, i}^\top=B_i U_{i+1} A_{i}^\top , \quad i\in [0, N-1]_\mathbb{N},
\end{aligned}
\end{equation}

Since $J_\K$ is a diagonal matrix with only zeros and ones on its diagonal, $\tilde{W}_\D$ will follow exactly the same pattern as $W_\D$. Denote $m_k = \sum_{i = 0}^{k-1} n_i$ for $k = [1, N]_\mathbb{N}$, $m_0 = 0$, and $J_{\K_i} \in \mathbb{R}^{n_i \times n_i}$ being a diagonal matrix whose diagonal elements equal the $m_i$ to $m_{i+1}$ elements of $J_\K$. Then we have
\begin{align}
\tilde{W}_\D=
&\left[\begin{array}{cccccc}
\tilde{W}_{00} & \tilde{W}_{01} & 0 & \cdots & 0 & 0 \\
\tilde{W}_{10} & \tilde{W}_{11} & \tilde{W}_{12} & \cdots & 0 & 0 \\
0 & \tilde{W}_{21} & \tilde{W}_{22} & \cdots & 0 & 0 \\
\vdots & \vdots & \vdots & \ddots & \vdots & \vdots \\
0 & 0 & 0 & \cdots & \tilde{W}_{N-1, N-1} & \tilde{W}_{N-1, N} \\
0 & 0 & 0 & \cdots & \tilde{W}_{N, N-1} & \tilde{W}_{N, N}
\end{array}\right],
  \end{align}
where 
\begin{align}
    \tilde {W}_{i,i} &= \alpha \beta J_{\K_i} W_{i,i} J_{\K_i} + I - J_{\K_i}, i \in [0, N]_\mathbb{N}\\
    \tilde{W}_{i, i+1}&= \tilde{W}_{i, i+1}^\top =  \alpha \beta J_{\K_i} W_{i,i+1} J_{\K_{i+1}}, i \in [0, N-1]_\mathbb{N}.
\end{align}

Next, we compute the Cholesky decomposition $\tilde{W}_\D = LL^\top$ using the factorization technique from \cite{wang2009fast}, where $L$ is lower triangular. The Cholesky factor $L$ has a lower bidiagonal block structure.

\begin{equation*}
L=\left[\begin{array}{cccccc}
L_{00} & 0 & 0 & \cdots & 0 & 0 \\
L_{01} & L_{11} & 0 & \cdots & 0 & 0 \\
0 & L_{21} & L_{22} & \cdots & 0 & 0 \\
\vdots & \vdots & \vdots & \ddots & \vdots & \vdots \\
0 & 0 & 0 & \cdots & L_{N-1, N-1} & 0 \\
0 & 0 & 0 & \cdots & L_{N, N-1} & L_{N, N}
\end{array}\right],
\end{equation*}
 where $L_{i, i}\in \mathbb{R}^{n_i\times n_i},$ $L_{i+1, i}\in \mathbb{R}^{n_{i+1}\times n_i}.$ Specifically, $L_{ii}$ are lower triangular matrices. Since $L L^\top=\tilde{W}_\D$, we have

\begin{equation}
\label{equ: newton factorization method}
\begin{aligned}
L_{00} L_{00}^\top & =\tilde{W}_{00} \\
L_{i, i} L_{i+1, i}^\top & =\tilde{W}_{i, i+1} \quad i\in \intset{0}{N-1}, \\
L_{i, i} L_{i, i}^\top & =\tilde{W}_{i i}-L_{i, i-1} L_{i, i-1}^\top, \quad i \in \intset{1}{N} .
\end{aligned}
\end{equation}

To compute $L_{00}$, we employ the Cholesky decomposition on $\tilde{W}_{00}$. Subsequently, $L_{10}$ is determined through forward substitution from the equation $L_{00} L_{10}^\top = \tilde{W}_{01}$. Next, $L_{11}$ is obtained by performing Cholesky decomposition on the matrix $\tilde{W}_{11} - L_{10} L_{10}^\top$. This procedure is repeated iteratively to construct the entire $L$ matrix.

We summarize the procedure for computing the Newton step in Algorithm \ref{algo: newtonstep for pipg}. It combines the insights from Lemma \ref{lemma: simplification} with the Cholesky factorization. Details of implementations are left to Section \ref{sec:Numerical implementation}.

\begin{algorithm}
\caption{Newton step for PIPG}
\label{algo: newtonstep for pipg}
\begin{algorithmic}[1]
\STATE \textbf{Require }  $H, \alpha, \beta, J_\D, J_\K, P, z^k, w^k$
\STATE \textbf{Compute } $R^k =T(z^k, w^k)-(z^k, w^k)$
\STATE \textbf{Compute } $\left(\begin{array}{c}
\tilde{R}_z^k \\
\tilde{R}^k_w
\end{array}\right)=\left[\begin{array}{cc}
I & 0 \\
-2 \beta H & I
\end{array}\right] R^k$
\STATE \textbf{Compute decomposition} $J_\D = Q_\D \Lambda_\D Q_\D^\top$ 
\STATE \textbf{Compute }$\bar R_{w}^k=J_{\K} \beta H V_\D \tilde R_z^k+\tilde{R}_w^k.$
\STATE \textbf{Compute matrix $\tilde{W}_\D$ using $Q_\D$ and $\Lambda_\D$}
\STATE \textbf{Compute $L$ for $\tilde{W}_\D$ using \eqref{equ: newton factorization method}}
\STATE \textbf{Compute } $$\Delta w=\left(L^\top\right)^{-1}  L^{-1}   (I-\alpha \beta^2 J_{\mathbb{K}^{\circ}} W_{\mathbb{D}}\left(I-J_{\mathbb{K}^{\circ}}\right)) \bar{R}_w^k$$
\STATE \textbf{Compute } $\Delta z=V_\D\left(\tilde{R}_z^k-\alpha J_\D H^{\top} \Delta w\right)$
\end{algorithmic}
\end{algorithm}


\section{Global Convergence for Newton-PIPG}
\label{sec:global convergence}
In this section, we introduce the Newton-PIPG algorithm, a hybrid of the PIPG operator and the Newton step.

We define the residual of the PIPG operator as $$
R(z,w) := T(z,w) - (z,w),
$$ and we denote $\operatorname{Fix}(T)$ the fixed point set of $T.$ Rather than employing the conventional Euclidean norm as the convergence criterion, we adopt an alternative norm:
$\|z\|_M:=\sqrt{\langle z, M z\rangle}$, where $M$ is a positive definite matrix expressed as
\begin{equation*}
M=\left[\begin{array}{cc}
\frac{1}{\alpha} I-P & -H^{\top} \\
-H & \frac{1}{\beta} I
\end{array}\right].
\end{equation*}
The positive definiteness of $M$ is proved later in Lemma \ref{thm: pd of M}. The associated inner product with this matrix is given by:
$$
\langle x,  y\rangle_M = \langle x, M y\rangle.
$$
and the norm $\| \cdot \|_M$ for a given matrix $A$ is defined as 
$$
\| A \|_M = \max_{\|x\|_M \leq 1}  \|Ax\|_M.
$$

The norm $\| \cdot \|_M$ is intrinsic to the proof of PIPG's convergence, rendering it a convenient choice for establishing convergence for the Newton-PIPG algorithm. Due to the equivalence among norms on finite-dimension problems, the convergence in $\| \cdot \|_M$ is equivalent to that in the Euclidean norm.

In order to ensure the convergence of the Newton method, we design Algorithm \ref{algo} that combines the PIPG with the Newton step, inspired by \cite{themelis2019supermann}. In this algorithm, the PIPG algorithm will be a safeguard to ensure global convergence, while the Newton step will be used to accelerate PIPG locally.
\begin{algorithm}
\caption{Newton-PIPG}
\label{algo}
\begin{algorithmic}[1]
\STATE \textbf{Require}  $c\in[0,1), e\geq0, t>0$, Initial iteration $(z^0, w^0)$
\STATE \textbf{Initialize} $k=0$
\WHILE{ termination criteria not satisfied }
\STATE \textbf{Compute} an update $p^k = (\Delta z^k, \Delta w^k)$ either be zero or from the Newton step

\IF{$$\left\|R(z^k + \Delta z^k, w^k + \Delta w^k)\right\|_M \leq c \left\|R (z^k, w^k)\right \|_M$$ and $$\|p^k\|_M < \sigma \|R(z^k, w^k)\|_M$$}

\STATE \textbf{Newton update :}$$(z^{k+1}, w^{k+1}) = (z^k + \Delta z^k, w^k + \Delta w^k)$$

\ELSE
\STATE \textbf{PIPG update :} $(z^{k+1},w^{k+1}) = \operatorname{T
}(z^{k}, w^{k})$
\ENDIF

\STATE k = k+1
\textit{}\ENDWHILE
\end{algorithmic}
\end{algorithm}

Algorithm \ref{algo} provides flexibility in choosing the update \( p^k \) while ensuring global convergence, as demonstrated in Theorem \ref{thm: convergence of PIPG}. In practice, \( p^k \) can be set to zero or to the solution of the Newton step. When \( p^k \) is zero, the Newton update is disabled. To achieve fast computation, we periodically activate the Newton step instead of using it at each iteration. Details on selecting \( p^k \) are available in Section \ref{sec:Numerical implementation}.

The termination criteria for Algorithm \ref{algo} is discussed in Section \ref{sec: termination criteria}. The coefficient $\sigma$ in Algorithm \ref{algo} is a safeguard ensuring that the Newton update does not move $(z^k, w^k)$ to a point far away from the current iteration. The use of such a coefficient $\sigma$ is inspired by \cite{themelis2019supermann}.


\subsection{Convergence Analysis of Newton-PIPG}

First, we show the positive definiteness of $M$.
\begin{lemma}
\label{thm: pd of M}
Assuming $\alpha, \beta > 0$, and $\alpha (\|P\| + \beta \|H\|^2) < 1$, $M$ is a positive definite matrix. 
\end{lemma}

\begin{proof}
Consider any vector $v=\left[\begin{array}{l}v_1 \\ v_2\end{array}\right],$ where $v_1 \in \mathbb{R}^{n_z}$ and $v_2 \in \mathbb{R}^{n_x(N+1)}.$ We have
$$
\begin{aligned}
\|v\|_M & =\left[v_1^{\top}, v_2^{\top}\right]\left[\begin{array}{cc}
\frac{1}{\alpha} I-P & -H^{\top} \\
-H & \frac 1 \beta I
\end{array}\right] \left[\begin{array}{l}
v_1 \\
v_2
\end{array}\right] \\
& =\frac{1}{\alpha}\left\|v_1\right\|_2^2+\frac{1}{\beta}\left\|v_2\right\|_2^2-2 v_2^{\top}  H  v_1-v_1^\top P v_1.
\end{aligned}
$$
Note that
$$
2 v_2^{\top} H v_1 \leq 2 \left\|v_1\right\|_2 \left\|v_2\right\|_2\|H\|_2 \leq \frac{1}{\beta}\left\|v_2\right\|_2^2+\beta\left\|v_1\right\|_2^2 \|H\|_2^2 .
$$
Thus $$\|v\|_M \geqslant \frac{1}{\alpha}\left\|v_1\right\|_2^2-\| P\|\left\| v_1\right\|_2^2-\beta \left\| H\right\|_2^2 \left\| v_1 \right\|_2^2\geqslant 0,$$ according to the assumption $\alpha, \beta > 0$ and $\alpha (\|P\| + \beta \|H\|^2) < 1$. The equality holds only when $v = 0.$\end{proof}

To simplify the notation, we set $\tilde z^k = (z^k, w^k),\ T\tilde z = T (z, w)$ and $R \tilde z = R(z, w).$ The following theorem, taken from \cite[Lemma 2]{yu2022extrapolated}, indicates a key property for the PIPG operator.
\begin{theorem}
\label{thm: convergence of PIPG}
Under the assumption in Lemma \ref{thm: pd of M}, for the PIPG operator $T,$ there exists a non-expansive operator $\tilde{T}$,  i.e. $\| \tilde T x - \tilde T y\|_M \leq \|x - y\|_M$, such that $T = \frac{1}{2} I+ \frac{1}{2} \tilde{T}$. 
\end{theorem}

We present the following theorem, which is a standard result concerning the convergence of algorithms based on nonexpansive operators. The proof, originally from \cite[Theorem 5.14]{bauschke2011convex}, is included here for completeness.

\begin{theorem}
\label{thm: km-convergence}
Suppose that $T$ satisfies  $T = \lambda I + ( 1 - \lambda) \tilde T$, where $\tilde T$ is nonexpansive, i.e. $\| \tilde T x - \tilde T y\|_M \leq \|x - y\|_M$. Considering an algorithm that $\tilde z^{k+1} = T\tilde z^{k}$. Then, the operator $R = I - T$ is continuous, and $\|R\tilde z^k\|_M$ is decreasing. Furthermore, if the fixed point set of $T$ is nonempty, the following hold: 
\begin{itemize}
    \item $\|\tilde z^{k+1}-y\|_M^2 \leq \|\tilde z^k-y\|_M^2-\lambda\left(1-\lambda\right)\|R \tilde z^k\|_M^2$ for any $y \in \operatorname{Fix}(T)$
    \item $\|R \tilde z^k\|_M$ converges to zero. 
    \item $\tilde z^k$ will converge to a fixed point of $T.$
\end{itemize}
\end{theorem}
\begin{proof}
We start with some facts related to $ T$ and $\operatorname{\tilde T}$. Firstly, $T$ is nonexpansive as
\begin{equation}
\begin{aligned}
\|T x-T y\|_M & =\|(1-\lambda)(x-y)+\lambda(\tilde{T} x-\tilde{T} y)\|_M \\
& \leq(1-\lambda)\|x-y\|_M+\lambda\|\tilde{T} x-\tilde{T} y\|_M \\
& \leq\|x-y\|_M .
\end{aligned}
\end{equation}Secondly, $y \in \operatorname{Fix}(T)$ if and only if $y \in \operatorname{Fix}(\tilde T).$ This is proved by plugging $y$ into $ T = \lambda I + ( 1 - \lambda) \tilde T$.

For the first conclusion, we know that, by the convexity of norm function, \[
\begin{aligned}
\|Rx - Ry\|_M &= \|\lambda(x - y) + (1 - \lambda)(\tilde T x - \tilde T y)\|_M \\
&\leq \lambda\|x - y\|_M + (1 - \lambda) \|\tilde T x - \tilde T y\|_M \\
&\leq \|x - y\|_M.
\end{aligned}
\]
Therefore, $R$ is a Lipschitz continuous operator. 

Meanwhile, since $T$ is nonexpansive,
$$\begin{aligned}\left\|R \tilde z^{k+1}\right \|_M & =\|\left.\tilde z^{k+1}-T\tilde z^{k+1}\left\|_M=\right \| T \tilde z^{k}-T\tilde z^{k+1}\|_M\right. \\ & \leq\left\|\tilde z^{k+1}-\tilde z^k\right \|_M=\left\|T \tilde z^k - \tilde z^k\right \|_M=\left\|R \tilde z^k\right \|_M.\end{aligned}$$ This proves the second half of the conclusion. 

Now we work on the other conclusions. For arbitrary $y$,
\begin{equation*}
\|\tilde z^{k+1}-y\|_M^2=\|\left(1-\lambda\right)\left(\tilde z^k-y\right)+\lambda\left(\tilde T \tilde z^k-y\right)\|_M^2.
\end{equation*}
We can expand and simplify this equation:
\begin{align*}
\|\tilde z^{k+1}-y \|_M^2 
&= \left\|(1-\lambda)(\tilde z^k - y) + \lambda(\tilde T \tilde z^k - y) \right\|_M^2 \\
&= (1-\lambda)\|\tilde z^k - y\|_M^2 + \lambda\|\tilde T \tilde z^k - y\|_M^2 
- \lambda(1-\lambda)\|\tilde z^k - \tilde T \tilde z^k\|_M^2.
\end{align*}

Here, we have used the fact that
\begin{equation*}
\begin{aligned}
\left\langle\ \tilde z^k-y,\ \tilde T \tilde z^k-y \right\rangle_M &= \frac{1}{2}\big(\left\|\tilde T \tilde z^k-y\right \|_M^2 + \left\|\tilde z^k-y\right \|_M^2 - \left\|\tilde z^k-\tilde T \tilde z^k\right \|_M^2\big) .
\end{aligned}
\end{equation*}

If $\operatorname{Fix}(T)$ is nonempty, we plug in arbitrary $y$ from this fixed point set. Since $y \in \operatorname{Fix}(\tilde T)$ as well, we have
\begin{equation*}
\left\|\tilde T \tilde z^k-y\right \|_M^2=\left\|\tilde T \tilde z^k-\tilde T y\right \|_M^2 \leq\left\|\tilde z^k-y\right \|_M^2.
\end{equation*}
Thus
\begin{equation*}
\begin{aligned}
\left\|\tilde z^{k+1}-y\right \|_M^2 &\leq \left\|\tilde z^k-y\right \|_M^2-\lambda\left(1-\lambda\right)\left\|\tilde T \tilde z^k-\tilde z^k\right \|_M^2 \\
&= \left\|\tilde z^k-y\right \|_M^2-\lambda \left\|R \tilde z^k\right \|_M^2.
\end{aligned}
\end{equation*}

A summation of both sides of this inequality for all iterations will show that $\|R\tilde z^k\|^2$ is summable and thus it will converge to zero. At the meantime $\left\|\tilde z^{k}-y\right \|_M^2$ is decreasing in $k$ and thus $\tilde z^k$ is bounded. We can therefore find a subsequence $\tilde z^{k_i}$ converging to a point $\tilde z^\star$ by compactness. As $\| R \tilde z^{k_i}\|_M$ convergence to zero and $R$ is continuous, we have that $T\tilde z^\star = \tilde z^\star$. Since $\tilde z^\star$ is a fixed point, we plug in $\tilde z^\star$ in the place of $y$ and get $\left\|\tilde z^{k+1}-\tilde z^\star\right \|_M^2 \leq\left\|\tilde z^k-\tilde z^\star\right \|_M^2$
for all $k$, and $\lim _i\left\|\tilde z^{k_i}-\tilde z^\star\right \|_M^2=0$. Thus
$
\lim_n \tilde z^k=\tilde z^\star
$
\end{proof}

We now proceed to establish the convergence of Newton-PIPG and start with the local convergence result.
\begin{theorem}
\label{thm:local_convergence}
With Assumptions \ref{assumption licq}, \ref{assum: symmetric Jacobian}, and \ref{assumption: interior-point}, consider the unique fixed point $\tilde{z}^\star$ of the operator $T$. There exists a neighborhood $K$ of $\tilde{z}^\star$ such that $\nabla R(\tilde{z}) = I - J_T(\tilde{z})$ for any $\tilde{z} \in K$. Consider vector $p$ for $\tilde{z} \in K$ and 
$$
\nabla R(\tilde{z})  p = -R(\tilde{z}).
$$
There exist constants $C_0$ and $C_1$ such that 
\begin{align}
\|R(\tilde{z} + p)\|_M &\leq C_0 \|R(\tilde{z})\|_M^2 \label{loc_eq:first} \\
\left\|\tilde{z} + p - \tilde{z}^\star\right\|_M &\leq C_1 \left\|\tilde{z} - \tilde{z}^\star\right\|_M^2 \label{loc_eq:second}
\end{align}
\end{theorem}

\begin{proof}
According to Theorem \ref{thm: nonsingularity}, there exists a compact neighborhood $K$ such that $\nabla R(\tilde{z}) = I - J_T$ exists and is invertible. Within this neighborhood, $\nabla R(\tilde z)$ is a smooth function and hence a $\gamma$-Lipschitz function for some constant $\gamma > 0$ with respect to the norm $\|\cdot\|_M$. The invertibility and smoothness of $\nabla R$ ensure that there exist constants $\mu> 0$  such that for any vector $\xi\in \MR^{n_z + n_w}$ and any $\tilde{z} \in K$, we have 
\begin{equation}
\label{local: equ_mu}
   \mu\|\xi\|_M \leq \|\nabla R(\tilde{z})  \xi\|_M . 
\end{equation}
Then for any $\tilde{z} \in K$ and $p \in \MR^{n_z +  n_w}$ such that $\nabla R(\tilde{z})  p = -R(\tilde{z})$, we have
\begin{equation*}
\begin{aligned}
R\left(\tilde z+p\right) & =R\left(\tilde z\right)+\int_0^1 \nabla R\left(\tilde z+\theta p\right)  p\ d \theta \\
& =\int_0^1\left(\nabla R\left(\tilde z+\theta p\right)-\nabla R\left(\tilde z\right)\right) p\ d \theta
\end{aligned}
\end{equation*}
Hence,
\begin{equation*}
\begin{aligned}
\| R(\tilde{z}+p) \|_M & \leq \int_0^1\|p\|_M \|\nabla R(\tilde z +  \theta p)-\nabla R(\tilde{z})\|_M d \theta \\
& \leq \int_0^1\|p\|_M \gamma\|\theta p\|_M d \theta=\frac{\gamma}{2}\|p\|_M^2 \\
&\leq \frac{\gamma}{2\mu^2}\|R(\tilde{z})\|_M^2,
\end{aligned}
\end{equation*}
where the last inequality is a consequence of $\|R(z)\|_M= \|\nabla R(\tilde{z}) p\|_M \geqslant\mu\|p\|_M.$ Therefore, \cref{loc_eq:first} is satisfied with $C_0 = \gamma/2\mu^2.$

Next, we show the validity of \eqref{loc_eq:second}. Setting $\eta = \tilde z - \tilde z^\star$ and using that $R(\tilde z^\star) = 0$, we have 
\begin{equation*}
\begin{aligned}
 R(\tilde z) &= R\left(\tilde z\right)-R\left(\tilde z^\star\right)
 =\int_0^1 \nabla R\left(\tilde z-\theta \eta\right)  \eta d \theta \\
=&\int_0^1\left(\nabla R\left(\tilde z - \theta \eta\right)-\nabla R\left(\tilde z\right))  \eta d \theta+\nabla R\left(\tilde z\right) \eta\right.
\end{aligned}
\end{equation*}

Hence
\begin{equation*}
\begin{aligned}
p & =-( \nabla R(\tilde{z}))^{-1} R(\tilde{z}) \\
& =-\eta-\int_0^1\left(\nabla R\left(\tilde z\right)\right)^{-1}\left(\nabla R\left(\tilde z-\theta \eta\right)-\nabla R\left(\tilde z\right)\right) \eta d \theta.
\end{aligned}
\end{equation*}

Using $\eta=\tilde{z}-\tilde z^\star$, we have
\begin{equation*}
\begin{aligned}
& \left\|\tilde{z}+p-\tilde{z}^\star\right\|_M \\
\leqslant & \int_0^1\left\|(\nabla R(\tilde{z}))^{-1} \left(\nabla R\left(\tilde z-\theta \eta\right)-\nabla R\left(\tilde z\right)\right) \eta\right\|_M d \theta \\
\leqslant &\mu^{-1} \int_0^1\left\|\nabla R\left(\tilde z-\theta \eta\right)-\nabla R\left(\tilde z\right)\right\|_M \|\eta\|_M d \theta \\
\leqslant &\mu^{-1} \int_0^1 \gamma  \theta\|\eta\|_M^2 d \theta=\frac{ \gamma}{2\mu}\left\|\tilde{z}-\tilde{z}^\star\right\|_M^2
\end{aligned}
\end{equation*}
which proves equality \eqref{loc_eq:second}. Here, the second inequality follows from \eqref{local: equ_mu}, while the last inequality is a result of the Lipschitz continuity of $\nabla R(\tilde{z})$.\end{proof}

Now we prove the global convergence of Newton-PIPG:
\begin{theorem}
\label{main_convergence}
Suppose the conditions of Lemma \ref{thm: pd of M} hold and \( \operatorname{Fix}(T) \) is non-empty. If the algorithm runs indefinitely, disregarding the termination criteria in line 3 of Algorithm \ref{algo}, the following properties hold:
\begin{enumerate}
    \item $R \tilde{z}^{k} \rightarrow 0$.
    \item Given Assumptions \ref{assumption licq}, \ref{assum: symmetric Jacobian}, and \ref{assumption: interior-point}, if $\sigma$ is sufficiently large and $p_k$ is computed using \eqref{equ: newtonstep} for each iteration unless $I - J_T$ is singular, Newton-PIPG converges to the unique fixed point $\tilde z^\star$ of $T$, with a quadratic local convergence rate.
\end{enumerate}
\end{theorem}

\begin{proof}
For the first result, consider a proof by contradiction. Given Theorem \ref{thm: km-convergence} and the restriction on the Newton step, $\|R \tilde{z}^k\|_M$ is monotonically decreasing. If $R \tilde{z}^k$ doesn't converge to zero, it implies that the Newton step is eventually disabled, leading us to the same algorithm as described in Theorem \ref{thm: km-convergence}. The convergence results from Theorem \ref{thm: convergence of PIPG} and Theorem \ref{thm: km-convergence} then imply a contradiction. Hence, $R \tilde{z}^k \rightarrow 0$. 

For the second result, let $\tilde{z}^\star$ be the unique fixed point of the PIPG operator. The uniqueness of the fixed point is guaranteed by extended-LICQ and the strong convexity of Problem~\hyperref[equ: original problem]{1}. To start, we aim to show that $\|\tilde{z}^k - \tilde{z}^\star\|_M$ is bounded.

In the $k$-th iteration, if the PIPG update is used, then 
$$\left\|\tilde{z}^{k+1} - \tilde{z}^\star\right\|_M \leq \left\|\tilde{z}^k - \tilde{z}^\star\right\|_M$$
as per Theorem \ref{thm: km-convergence}. If the Newton update is accepted, the criteria in line 5 of Algorithm \ref{algo} ensures that
\begin{equation*}
\begin{aligned}
\left\|\tilde{z}^{k+1} - \tilde{z}^\star\right\|_M & - \left\|\tilde{z}^k - \tilde{z}^\star\right\|_M \leq \left\|\tilde{z}^{k+1} - \tilde{z}^k\right\|_M \\
& \leq \sigma \left\|R \tilde{z}^k \right\|_M \leq \sigma  c^{q_k} \left\|R \tilde{z}^0\right\|_M.
\end{aligned}
\end{equation*}

Here, $q_k$ is the number of times the Newton step has been activated since the first iteration. Coefficients $c$ and $\sigma$ are defined in Algorithm \ref{algo}.

Summing over $k$, we obtain 
$$\left\|\tilde{z}^n - \tilde{z}^\star\right\|_M \leq \left\|\tilde{z}^0 - \tilde{z}^\star\right\|_M + (\sigma \sum_{i=1}^\infty c^i) \left\|R \tilde{z}^0\right\|_M.$$
Given that $c < 1$, we infer that $\tilde{z}^k$ is bounded.

According to Theorem \ref{thm:local_convergence}, there exists a neighborhood $K$ such that equations \eqref{loc_eq:first} and \eqref{loc_eq:second} hold. Therefore, we can find an $M$-norm ball $K' \subset K$ such that for any $\tilde{z} \in K'$ and vector $p$ such that $\nabla R(\tilde{z})  p_k = -R(\tilde{z})$, we have 
\begin{equation*}
\begin{aligned}
\|R(\tilde{z} + p_k)\|_M & \leq c \|R(\tilde{z})\|_M \\
\left\|\tilde{z} + p_k - \tilde{z}^\star\right\|_M & \leq \left\|\tilde{z} - \tilde{z}^\star\right\|_M,
\end{aligned}
\end{equation*}
where $c$ is defined in Algorithm \ref{algo}. Hence, for any $\tilde{z}^k \in K'$, we have $\tilde{z}^k + p_k \in K'$. 

Therefore, if $\sigma$ is sufficiently large, specifically $\sigma \geq \sup_{z \in K} \|\nabla R^{-1}\|_M$, we will have $\|p_k\|_M \leq \sigma \|R(\tilde{z}^k)\|_M$ for all $\tilde z^k \in K$. 

By the compactness of $\tilde{z}^k$, the continuity of $R$, and that $R \tilde{z}^k \rightarrow 0$, there exists a converging subsequence of $\tilde{z}^k$ that converges to the unique fixed point $\tilde{z}^\star$. Then, when $\tilde{z}^k \in K'$, the Newton step will be accepted as it leads to a sufficient amount of decrease in $R(\tilde{z}^k)$, and $\|p_k\|_M \leq \sigma \|R(\tilde{z}^k)\|_M$. Further, $\tilde{z}^{k+1} \in K'$, ensuring that the PIPG will not be activated again. In that case, \eqref{loc_eq:second} in Theorem \ref{thm:local_convergence} ensures a quadratic convergence rate.\end{proof}

\begin{remark}
\label{rm: convergence within finite iter}
The quadratic convergence result in Theorem \ref{main_convergence} can be further strengthened to convergence within finite iterations if $J_\D$ and $J_\K$ are piecewise constant functions. This scenario occurs when $J_\D$ consists of boxes, halfspaces, and affine subspaces. In these cases, $J_T$ will also be a piecewise constant function of the primal and dual variables. Within a local neighborhood of the fixed point $(z^\star, w^\star)$, the PIPG operator will be a linear function, allowing the Newton step to find the fixed point within one iteration.
\end{remark}

\begin{remark}
The proof method in Theorem \ref{main_convergence} is inspired by \cite{themelis2019supermann}. The difference between the proof in Theorem \ref{main_convergence} and the convergence analysis in \cite{themelis2019supermann} is that we establish convergence for a hybrid approach combining the Newton method, instead of the quasi-Newton method, with the operator splitting method, thus achieving a better convergence rate. Furthermore, the invertibility of \( I - J_T \) guarantees the existence of the coefficient \( \sigma \) in our proof (corresponding to \( D \) in \cite[Assumption 2]{themelis2019supermann}), eliminating the requirement of \cite[Assumption 2]{themelis2019supermann}, which is not commonly used in optimization research.
\end{remark}

\section{Restrictions and Relaxation}
\label{sec: application and relaxation}
In this section, we point out several restrictions in Problem~\hyperref[equ: PIPG problem]{2} compared to general optimal control quadratic programming problems, and we discuss the possibility of relaxing such restrictions.

In Problem~\hyperref[equ: PIPG problem]{2}, we require $P$ to be a diagonal matrix with positive diagonal elements. This requirement ensures the invertibility of $I - J_\D + \alpha J_\D P$, which is necessary for the linear algebra technique outlined in Lemma \ref{lemma: simplification}. Although this choice of $P$ is suitable for many optimal control problems, it does not encompass all scenarios. In particular, it does not cover QPs arising in the sequential convexification of nonconvex optimal control problems that utilize the $L_1$-exact penalty \cite{malyuta2022convex, szmuk2020successive, elango2024successive}. When applying the $L_1$ exact penalty, transforming the linearized subproblems into QP form requires the introduction of slack variables. Consequently, the resulting QP no longer has a positive definite quadratic term; the quadratic matrix contains zeros on the diagonal for positions that correspond to these slack variables. To address this issue, we propose adding an additional squared $L_2$-norm to the penalty. This modification not only maintains the exactness but also ensures our algorithm's compatibility with the subproblems from the sequential convexification.

Besides the form of $P$, we also require $J_\D = J_\D ^\top$, $Q_\D P = P Q_\D$, and $\mathbb{K} = \prod_{i=0}^N (0^{n_{E_i}} \times \mathbb{R}_{+}^{n_{I_i}})$. Generalizing these constraints introduces several challenges. If $J_\D$ is not symmetric or $Q_\D P \neq P Q_\D$, the relationship $\left(I - J_\D (I - \alpha P)\right)^{-1} = V_\D$ in Lemma \ref{lemma: simplification} will no longer hold. Furthermore, if we generalize $\mathbb{K}$ to a general cone, $J_\K (I - J_\K) = 0$ will not be satisfied. In either case, the linear matrix appearing in the Newton step will not be symmetric. Consequently, the Cholesky decomposition will not be applicable, and the proof for the nonsingularity of $I - J_T$ will be invalid.

If generalization is necessary, one can directly compute the factorization for \(I - J_T\) at some Newton steps and use iterative methods such as GMRES\cite{saad1986gmres} for other Newton steps. In this approach, previous factorizations can be used as preconditioners for subsequent iterations.

\section{Implementation and Numerical Experiments}
\label{sec:Numerical implementation}
We demonstrate the implementation details of the Newton-PIPG algorithm and compare its performance with other state-of-the-art solvers. 

\subsection{Implementation Details for Newton-PIPG}
Algorithm \ref{algo} provides a framework for the Newton-PIPG algorithm, allowing flexibility in choosing the form of \(p^k\). This flexibility leads to different strategies for implementing the Newton-PIPG algorithm. Here are some strategies we use to optimize computation speed.

First, we reduce the use of Newton steps by applying them only when PIPG iterations stagnate. A heuristic waiting period, typically 3-10 PIPG iterations, is employed. If both \(J_\K\) and the active manifold of \(\pi_\D\) remain unchanged during this period, the algorithm switches to the Newton step.

Second, after each computationally intensive Newton step, a line search is performed to monitor the reduction in the PIPG residual \(R(z^k, w^k)\). This line search, typically executed 3-5 times per Newton step, improves practical convergence speed without affecting the convergence properties of the Newton-PIPG algorithm. The line search result that reduces the residual norm by a factor of \(c\), as defined in Algorithm \ref{algo}, is accepted. The coefficient \(c\) is generally close to 1; for instance, we use \(c = 0.99\). If the line search fails to sufficiently reduce the residual, the result is discarded, and a PIPG update is performed instead. Moreover, in the case of a failed line search, the waiting period is extended until either \(J_\K\) or the active manifold of \(\pi_\D\) is updated.

Third, to improve the conditioning of the optimization problem, we rescale the matrix \(H\) so that the norm of each row is of the same magnitude. This rescaling does not affect the optimal solution. For a detailed discussion, we refer interested readers to \cite{kamath2023customized}.

Fourth, \(I-J_T\) can become singular when the current iteration is not in a local neighborhood of the primal-dual solution pair \((z^\star, w^\star)\). To obtain useful results from the Newton step, we add an additional perturbation to the matrix \(I-J_T\). The perturbation we choose is an identity matrix multiplied by a coefficient proportional to the norm of the residual \(R(z^k, w^k)\), so that the perturbation becomes negligible when the current iteration is close to \((z^\star, w^\star)\).

Fifth, we monitor the convergence of the residual using the Euclidean norm instead of \(\|\cdot\|_M\), as the Euclidean norm is more efficient to compute. In our experiments, using the Euclidean norm did not compromise the algorithm's convergence. However, in situations where it might, one may consider using \(\|\cdot\|_M\) to validate the Newton step for a more robust implementation.

Finally, we optimize the matrix-vector multiplication by partitioning the sparse matrix \(H\) into zero blocks and dense, nonzero blocks. Instead of performing sparse matrix multiplication, we treat \(H\) as a collection of dense block matrices. Using for-loops, we compute the matrix-vector product as a series of dense matrix-vector multiplications. This approach is also applied when computing the factorization in the Newton step. A similar method is used in \cite{yu2020proportional}.

\subsection{Termination Criteria}
\label{sec: termination criteria}
In this section, we propose the termination criteria for both PIPG and Newton-PIPG. We claim that monitoring the magnitude of $\|R(z^k, w^k)\| = \|(z^{k+1}, w^{k+1}) - (z^k, w^k)\|$ provides a good criterion for stopping the PIPG algorithm. Here, $(z^k, w^k)$ and $(z^{k+1}, w^{k+1})$ are consecutive PIPG iterations. Since Newton-PIPG requires iteratively calling the PIPG operator, we can use the same termination criteria for the Newton-PIPG algorithm as well.

To start, we show that $\|(z^{k+1}, w^{k+1}) - (z^k, w^k)\|$ is a good indicator of the quality of $(z^{k+1}, w^{k+1})$ in terms of the KKT condition in Theorem \ref{thm: extended-licq}. Specifically, we set $\Delta z = z^{k+1} - z^k$ and $\Delta w = w^{k+1} - w^k$ and show that:
\begin{theorem}
\label{thm:stop_criteria}
When both $\|\Delta z\| \leq \epsilon$ and $\|\Delta w\| \leq \epsilon$ for $\epsilon > 0$, we have the following relationship:
\begin{equation*}
\begin{aligned} d_{N_\D\left(z^{k+1}\right)}\left(-P z^{k+1}- q - H^\top w^{k+1}\right)&\leq \left(\frac{1}{\alpha} + \|P\| + \|H^\top\|\right) \epsilon \\
d_{N_\K\left(w^{k+1}\right)}\left(H z^{k+1} - g\right) &\leq \left(\frac{1}{\beta} + \|H\|\right) \epsilon\\
 z^{k+1} \in \D &\ and \ w^{k+1} \in \K  \\
\end{aligned}
\end{equation*}
\end{theorem}

\begin{proof}
The PIPG iterations \eqref{PIPG iter} trivially guarantees 
$z^{k+1} \in \D \ and \ w^{k+1} \in \K$. The PIPG iteration \eqref{PIPG iter} ensures that 
\begin{equation*}
\begin{aligned}
z^{k+1}=&\pi_\D[z^{k+1}-\Delta z - \alpha P(z^{k+1}-\Delta z)-\alpha\left(q+H^{\top}\left(w^{k+1}-\Delta w\right))\right] \\
 w^{k+1}=&\pi_{\K}\left[w^{k+1}-\Delta w+\beta\left(H\left(z^{k+1}+\Delta z\right)-g\right)\right]
\end{aligned}
\end{equation*}  
Using properties of projections to a convex set, we obtain
\begin{equation*}
\begin{aligned}
& P(\Delta z-{z^{k+1}})-q+H^{\top} (\Delta w - w^{k+1})-\frac{\Delta z }{\alpha}\in N_\D\left(z^{k+1}\right) \\
& H z^{k+1}-g+H \Delta z-\frac{1}{\beta} \Delta w \in N_{\K}\left(w^{k+1}\right).
\end{aligned}
\end{equation*}
Using the fact that $\|\Delta z\| \leq \epsilon$ and $\|\Delta w\| \leq \epsilon$, we achieve
\begin{equation*}
\begin{aligned}
d_{N_\D\left(z^{k+1}\right)}(-P z^{k+1}-q-H^T w^{k+1})&\leq \|P\Delta z+H^{\top} \Delta w-\frac{\Delta z }{\alpha}\|\\
&\leq\left(\frac{1}{\alpha}+\|P\|+\| H^{\top} \|\right) \varepsilon \\
\left.d_{N_\K\left(w^{k+1})\right.}\right.\left(H z^{k+1}-g\right)
&\leq \| H \Delta z -\frac{\Delta w}{\beta}\| \leq\left(\frac{1}{\beta}+\|H\|\right) \varepsilon.
\end{aligned}
\end{equation*}
\end{proof}

Setting $\gamma_p =  \frac{1}{\alpha} + \|P\| + \|H^\top\|$ and $\gamma_d = \frac{1}{\beta} + \|H\|$, Theorem \ref{thm:stop_criteria} indicates that $\gamma_p\|\Delta z\|$ and $\gamma_d\|\Delta w\|$ measure the KKT satisfaction for the current iteration $(z^k, w^k)$. Hence, we propose the following termination condition for both PIPG and Newton-PIPG: given an absolute tolerance $\epsilon_{abs}$ and a relative accuracy tolerance $\epsilon_{rel}$, the algorithm will be terminated if
\begin{equation*}
\begin{aligned}
\left\|z^{k+1} - z^k\right\| &\leq \frac{\epsilon_{\text{abs}} + \epsilon_{\text{rel}} \left\|P z^{k+1}+q+H^{\top} w^{k+1}\right\|}{\gamma_p}\label{equ: termination}\\
\left\|w^{k+1} - w^k\right\| &\leq \frac{\epsilon_{\text{abs}}+ \epsilon_{\text{rel}}\|H z^{k+1}-g\|}{\gamma_d}.   
\end{aligned}
\end{equation*}


\subsection{Numerical Experiment Introduction}
We test the performance of the Newton-PIPG algorithm via two example problems. The first example is an oscillating masses problem that involves only linear constraints. The second example is a powered-descent guidance problem that includes various types of constraints, such as linear inequality, ball constraints and second-order cones. We compare the results against benchmark algorithms such as ECOS \cite{domahidi2013ecos, andersen2000mosek} and OSQP \cite{stellato2020osqp}, among others.

The numerical experiment is conducted using MATLAB. The Newton-PIPG code was initially written in MATLAB and then converted into executable C code using MATLAB Coder. The computational experiments were carried out on a computer equipped with an AMD Ryzen 7 5700G 8-Core CPU and 16GB of RAM. The results here are generated with code in: \scalebox{0.85}{\href{https://github.com/UW-ACL/NEWTON-PIPG-for-QP-Problems}{\texttt{github.com/UW-ACL/NEWTON-PIPG-for-QP-Problems}}}.

\subsection{Oscillating Masses}
We use a benchmark oscillating masses problem \cite{wang2009fast,jerez2014embedded} to evaluate the performance of our algorithm in comparison to other state-of-the-art algorithms. Consider a mechanical system comprising $m = 16$ masses connected in a one-dimensional line, shown in figure \ref{fig: mass-spring model}. Each mass is defined by its 1-D velocity and position, resulting in a total of $n_x = 16$ state variables. To control this mechanical system, we assign one control variable to each mass. Hence the dimension of control at each time point is $n_u = 8.$ Time is discretized using a step size of $\Delta$. The total number of time points is denoted by N, and we set $\Delta = 3/N,$ i.e. this equation describes a system among $[0,3]$. In our experiment, we set $N = 20, 50, 100$ to show the scalability of our algorithm. 
\begin{figure}[!ht]
\centering
\begin{adjustbox}{scale=0.6}

\begin{circuitikz}

\pattern[pattern=north east lines] (-0.2, -0.8) rectangle (0, 0.8);
\draw[thick] (0, -0.8) -- (0, 0.8);
\draw (0, 0) to[spring] (1, 0);
\draw (1, -0.5) rectangle (2, 0.5);
\draw (2, 0) to[spring] (3, 0);
\draw (3, -0.5) rectangle (4, 0.5);
\draw (4, 0) to[spring] (5, 0);

\node at (5.5, 0) {\huge $\ldots$};

\draw (6, 0) to[spring] (7, 0);
\draw (7, -0.5) rectangle (8, 0.5);
\draw (8, 0) to[spring] (9, 0);
\draw (9, -0.5) rectangle (10, 0.5);
\draw (10, 0) to[spring] (11, 0);

\pattern[pattern=north east lines] (11, -0.8) rectangle (11.2, 0.8);
\draw[thick] (11, -0.8) -- (11, 0.8);
\end{circuitikz}

\end{adjustbox}
\caption{The oscillating masses system}
\label{fig: mass-spring model}
\end{figure}
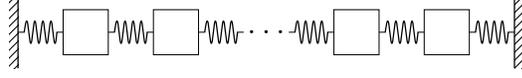

After discretization, we have the following optimization problem:
\begin{equation*}
\label{numerical problem}
\begin{array}{ll}
\underset{u, x}{\operatorname{minimize}} & \frac{1}{2} \sum_{i=0}^{N+1} x_i^{\top}P x_i+\frac{1}{2} \sum_{i=0}^{N} u_i^{\top} Q u_i \\
\text { subject to } & x_{i+1}=A x_i+B u_i,\quad i\in \intset{0}{N},\\
& u_{min} \leq u_i \leq u_{max}, \quad i\in \intset{0}{N} \\
&x_{min} \leq x_i \leq x_{max}, \quad i\in \intset{0}{N+1}\\
& x_0=\hat{x}_0.
\end{array}
\end{equation*}
Here, the matrices $P$ and $Q$ are identity matrices and matrices $A$ and $B$ are defined as
$$
\begin{aligned}
& A=\exp \left(\Delta\left[\begin{array}{cc}
0_{m \times m} & I_m \\
-\mathbb{L} & 0_{m \times m}
\end{array}\right]\right), \\
& B=\int_0^{\Delta} \exp \left(s\left[\begin{array}{cc}
0_{m \times m} & I_m \\
-\mathbb{L} & 0_{m \times m}
\end{array}\right]\right) \mathrm{d} s\left[\begin{array}{c}
0_{m \times m} \\
I_m
\end{array}\right],
\end{aligned}
$$
where the matrix $\mathbb{L} \in \mathbb{R}^{m \times m}$ is a tridiagonal matrix with 2's on the diagonal and -1's on the subdiagonal and superdiagonal entries. 

The initial conditions $\hat{x}_0$ are randomly generated. In our experiments, we set $x_{\text{min}} = -1$ and $x_{\text{max}} = 1$. The initial condition is generated from a normal distribution with mean zero and a standard deviation of 0.3 then projected to $[x_\text{min}, x_\text{max}]$. This ensures that the values of $x$ are spread within the range of $[-1, 1]$, without too many values equal to -1 or 1.

In our numerical experiments, we compared our algorithm with several prominent solvers: OSQP \cite{stellato2020osqp}, SCS \cite{odonoghue:21}, SCS-Anderson \cite{zhang2020globally}, PIPG \cite{yu2022extrapolated}, and ECOS \cite{domahidi2013ecos}, using their respective MATLAB packages. OSQP and SCS are ADMM-based methods. SCS-Anderson represents methods that combine operator-splitting operators with quasi-Newton methods. Lastly, ECOS stands as an example of an interior-point method.

Since our current algorithm does not have feasibility detection capacity, we only consider problems with feasible solutions. In practice, we run OSQP and SCS before the Newton-PIPG and Newton-PIPG will only be used to solve problems that are characterized as feasible by both OSQP and SCS. 

%
For our numerical experiments, we set the tolerance level of all competing solvers to $10^{-8}$ in their respective software settings, and for PIPG we set $\epsilon_{abs} = 10^{-8},$ and $\epsilon_{rel} = 0$. The Newton-PIPG algorithm does not require a specific termination criterion, as it converges within a finite number of iterations when the optimization problem is linear. Post-computation, we measure and report the difference between each solver's solution and that of Newton-PIPG using the $L_2$ norm.

 We set $(u_{\text{min}}, u_{\text{max}}) \in \{(-1, 1), (-0.4, 0.4)\}$ and then randomly generate 100 initial conditions for each combination of $N$ and $(u_{\text{min}}, u_{\text{max}})$. The average time consumption is shown in Table \ref{table: numerical result}, measured in milliseconds. In this table, the SCS-And refers to the SCS-Anderson algorithm, while the N-PIPG denotes the Newton-PIPG. These time consumption are reported by the Matlab package of each solver. When the maximum control magnitude is set to 1, we observed that all randomly generated problems are feasible. When the control magnitude is set to 0.4, approximately 95\% of the randomly generated problems are feasible, and we only apply the Newton-PIPG algorithm to these problems.

When set to $(u_{\text{min}}, u_{\text{max}}) =(-1 ,1)$ and $N = 20$, our algorithm outperforms others by a factor of five. Such a ratio persists with larger problem sizes. With respect to the post-computation error tolerance, ECOS attains $10^{-5}$ in $L_2$ norm in all cases. At $N= 20,$ all other solver achieves $10^{-8}$ in accuracy. At $N = 50$ and $N = 100$, algorithms other than PIPG and ECOS achieve accuracy of $10^{-6}$. PIPG reaches the target accuracy for all cases.

The Newton-PIPG algorithm achieves rapid QP solution primarily because the PIPG step rapidly identifies the correct active constraints of $\mathbb{D}$ in situations where the constraints are not tight. Thus, it usually requires a single Newton step and a small number of PIPG steps to solve the problem. To evaluate scenarios where more Newton steps are needed, we set $(u_{\text{min}}, u_{\text{max}}) =(-0.4 ,0.4)$. With more inequality constraints activated, more Newton steps are expected. As predicted, Newton-PIPG slows down but is still 30\% faster than the best of the other algorithms.

We also compare Newton-PIPG with SCS-Anderson, as both methods hybridize operator-splitting and second-order approaches. SCS-Anderson does speed up SCS, but the increase is limited. However, adding the Newton step to PIPG dramatically improves its speed, both at $u = 1$ and $u = 0.4$. This indicates that the Newton step has better performance compared to quasi-Newton methods. 

We presented a typical plot (Figure \ref{fig:newtonpipgvspipg}) comparing the computation time between the Newton-PIPG method and the PIPG method. The plot tracks the computation time and the norm of the residual after each PIPG and Newton step, with \(N = 20\) and \((u_{\text{min}}, u_{\text{max}}) = (-0.4, 0.4)\). The y-axis represents the logarithm of the norm of the difference between two consecutive iterations, and the x-axis represents the solve time in milliseconds. Figure \ref{fig:newtonpipgvspipg} shows that the Newton step is activated twice and is more efficient compared to PIPG iterations.

\begin{figure}[htbp]
    \centering
      \includegraphics[width=0.8\textwidth]{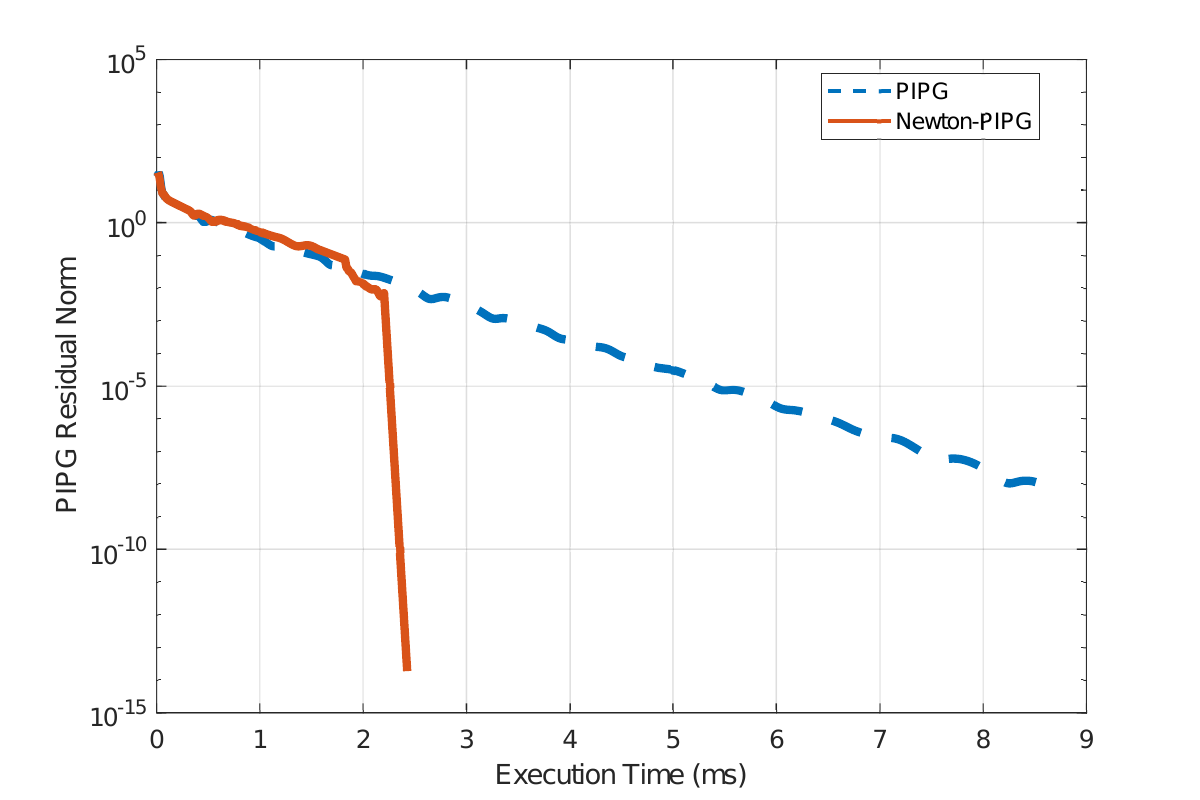}
    \caption{Comparison of residuals versus solve time for Newton-PIPG and PIPG algorithms for the oscillating masses example.}
    \label{fig:newtonpipgvspipg}
\end{figure}

\begin{table*}[htbp]
\centering
\caption{Performance for different method on the oscillating mass example (All times in milliseconds.)}
\label{table: numerical result}
\begin{tabular}{cccccccc}
\hline
\textbf{Umax} & \textbf{N} & \textbf{SCS-And} & \textbf{OSQP} & \textbf{SCS} & \textbf{ECOS} & \textbf{PIPG} & \textbf{N-PIPG} \\
\hline
1    & 20  & 3.48  &  {3.47}  & 3.93 & 16.84 & 9.42  & \textbf{0.61} \\
0.4  & 20  & 6.97  &  {3.74}  & 7.22 & 17.30 & 8.75  & \textbf{1.53} \\
1    & 50  & 14.88 &  {11.48} & 14.25 & 40.80 & 33.73 & \textbf{2.10} \\
0.4  & 50  & 21.02 &  {10.41} & 20.96 & 40.91 & 43.14 & \textbf{7.68} \\
1    & 100 &  {46.72} & 72.82 & 73.11 & 83.45 & 149.36 & \textbf{4.88} \\
0.4  & 100 & 51.50 &  {25.47} & 52.71 & 84.05 & 161.99 & \textbf{16.28} \\
\hline
\end{tabular}
\end{table*}

\subsection{Powered-Descent Guidance Problem}
Our second numerical example considers a powered-descent guidance (PDG) problem for soft-landing a rocket-powered vehicle on a planetary body. Specifically, we consider the reference trajectory tracking problem from \cite{elango2022customized}. This problem can be formulated to fit within the form of Problem~\hyperref[equ: original problem]{1} and can be solved using the Newton-PIPG method. We use this problem to demonstrate the efficiency of Newton-PIPG on complex problems with various types of constraints and to compare Newton-PIPG with interior-point solvers, which are the only methods among those used in the previous numerical example capable of directly solving such problems.

In the PDG problem, the rocket-powered vehicle is initially located at \(r_\text{init} \in \mathbb{R}^3\) with an initial velocity \(v_\text{init} \in \mathbb{R}^3\). The goal is to land this vehicle at the origin with a final velocity of zero. There are seven state variables describing the problem at each time point: three variables for location, three variables for velocity, and one variable for the log of the vehicle mass. Meanwhile, there are four control variables at each time point: three variables for a three-dimensional thrust and one slack variable to implement the lossless convexification technique. In our numerical problem, we choose the amount of time points as 30.

Constraints of the PDG problem include second-order cone constraints on the four-dimensional control variable and the location variables separately, ball constraints on velocity, linear equalities for the dynamics, linear inequalities for the relationship between log-mass and thrust, and box constraints on log-mass. For the exact form and discussion of this problem, refer to \cite[Problem 3]{elango2022customized}.

Compared to the original form of \cite[Problem 3]{elango2022customized}, we made several modifications. First, since some dimensions of the optimal solution are a few magnitudes larger than others, we adjusted the quadratic coefficients in the objective function to ensure that each dimension of the optimal solution contributes similarly to the objective function. This adjustment is primarily to enhance the performance of the interior-point method. Without this change, the solution obtained by the interior-point method exhibited large relative errors on some dimensions of the solution, making the comparison between the interior-point method and Newton-PIPG less meaningful.

Secondly, we retained most of the coefficients listed in \cite[Table 3]{elango2022customized}, modifying only \(r_\text{init}\), the initial location of the vehicle. We set \(r_{\text{init}_{(1)}} = 0\, \text{m}\)  and \(r_{\text{init}_{(3)}} = 2000\, \text{m}\), while \(r_{\text{init}_{(2)}}\) follows an arithmetic sequence from 0 m to 2900 m, with a common difference of 50 m. The optimization problem becomes infeasible when \(r_{\text{init}_{(2)}}\) exceeds 2950 m. As the second component of \(r_\text{init}\) increases, the problem approaches infeasibility, allowing us to test our algorithm in both scenarios where feasibility is easily guaranteed and where the problem is nearly infeasible.

In the numerical experiment, we compared the performance of Newton-PIPG with ECOS and PIPG. To use ECOS, we employed Yamlip \cite{Lofberg2004} to transform the original problem into a cone programming problem where ECOS is applicable. For the termination conditions, the termination criteria were set to \(\epsilon_{\text{abs}} = 10^{-12}\) and \(\epsilon_{\text{rel}} = 0\) for Newton-PIPG. Similarly, we set the accuracy to \(10^{-12}\) for ECOS. The post-computation error for ECOS, compared to the Newton-PIPG result, ranged from \(3 \times 10^{-6}\) to \(5 \times 10^{-7}\) for different choices of \(r_\text{init}\). For PIPG, we used two different termination criteria: \(\epsilon_{\text{abs}} = 10^{-4}\) and \(\epsilon_{\text{abs}} = 10^{-8}\), with \(\epsilon_{\text{rel}} = 0\). The computed errors for PIPG showed that the algorithm achieved the assigned accuracy.

The wall time of the three algorithms under different initial conditions is shown in Figure \ref{fig:computation_time}, with a logarithmic time axis. In this figure, the line labeled "PIPG low accuracy" represents the results of the PIPG algorithm with a tolerance of \(10^{-4}\), while "PIPG high accuracy" corresponds to a tolerance of \(10^{-8}\). For \(r_\text{init}\) values greater than 2250 meters in the second dimension, PIPG did not converge within 50,000 iterations. As a result, we used hollow markers to represent these non-converging points, which are not explicitly mentioned in the legend.
The line labeled "ECOS" shows the wall time reported by Yamlip, excluding Yamlip's compile time.

\begin{figure}[htbp]
    \centering
    \includegraphics[width=0.8\textwidth]{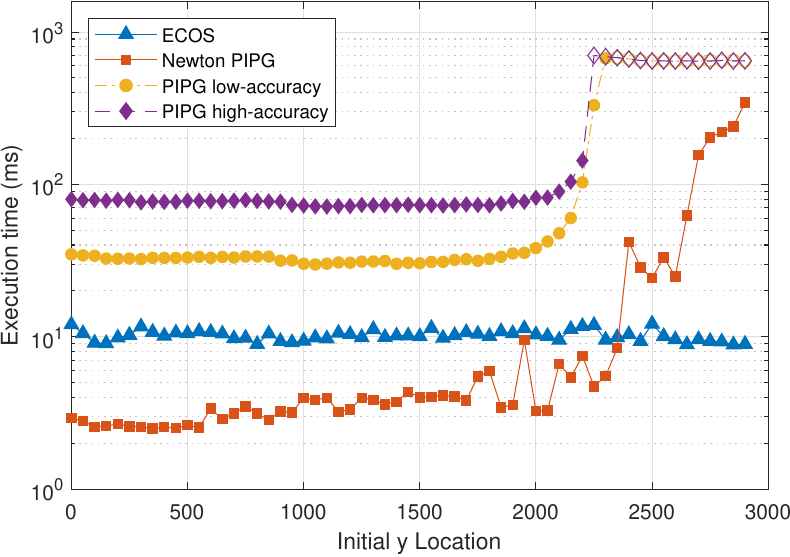}
    \caption{Comparison of execution times for ECOS, Newton-PIPG, and PIPG with low accuracy (tolerance \(10^{-4}\)) and high accuracy (tolerance \(10^{-8}\)) across different initial y locations. Hollow markers represent non-converged points, which are omitted from the legend.}
    \label{fig:computation_time}
\end{figure}

Comparing the computation speed between Newton-PIPG and ECOS, we observed that Newton-PIPG is notably faster than ECOS when the problem is not close to infeasibility. However, as the problem approaches infeasibility, especially when the second dimension of $r_\text{init}$ is larger than 2400 meters, the computation time of Newton-PIPG increases dramatically, while the speed of the interior-point methods appears unaffected by the proximity to infeasibility. Therefore, our numerical experiments suggest that Newton-PIPG is preferred when a highly accurate solution is required or when the problem is not close to infeasibility.

Figure \ref{fig:solvetime_vs_residual} compares the residuals of Newton-PIPG and PIPG versus wall time, with \( r_\text{init} = (0, 0, 2000) \) meters, using the same method as Figure \ref{fig:computation_time}. Newton-PIPG requires six Newton steps to converge below the threshold of \( 10^{-12} \), with the last three Newton steps contributing most to the significant decrease in the residual norm, consistent with theoretical expectations.
\begin{figure}[htbp]
    \centering
    \includegraphics[width=0.8\textwidth]{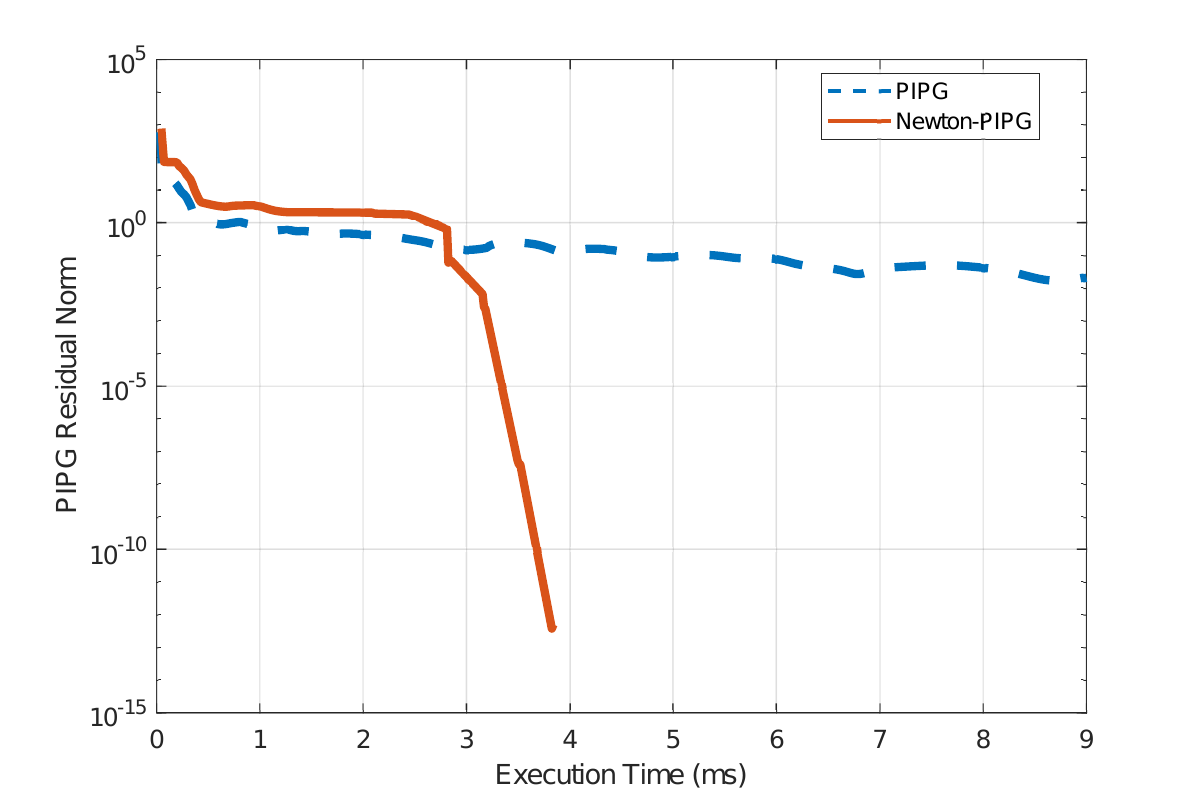}
    \caption{Comparison of residuals versus solve time for Newton-PIPG and PIPG algorithms for the PDG example, with \( r_\text{init} = (0, 0, 2000) \) meters.}
    \label{fig:solvetime_vs_residual}
\end{figure}

\section{Conclusion and Future Work}
In conclusion, we introduced the Newton-PIPG method for solving optimal control QP problems, which combines an operator splitting method, PIPG, with second-order Newton steps. We demonstrated the convergence of this algorithm and provided an efficient technique for solving the linear system in the Newton step. Our numerical experiments showed that our algorithm performs well compared to other state-of-the-art algorithms in solving quadratic optimal control problems.

For future research, we will incorporate infeasibility detection for Newton-PIPG and extend our algorithm to handle more general constraints. Additionally, the current Newton-PIPG software is written in Matlab, and we expect that a pure C/C++ implementation could further reduce computation time.





\bibliographystyle{siamplain} 

\bibliography{refs}%

\begin{thebibliography}{10}

\bibitem{accikmecse2011robust}
{\sc B.~A{\c{c}}{\i}kme{\c{s}}e, J.~M. Carson~III, and D.~S. Bayard}, {\em A
  robust model predictive control algorithm for incrementally conic
  uncertain/nonlinear systems}, International Journal of Robust and Nonlinear
  Control, 21 (2011), pp.~563--590.

\bibitem{ali2017semismooth}
{\sc A.~Ali, E.~Wong, and J.~Z. Kolter}, {\em A semismooth newton method for
  fast, generic convex programming}, in International Conference on Machine
  Learning, PMLR, 2017, pp.~70--79.

\bibitem{andersen2000mosek}
{\sc E.~D. Andersen and K.~D. Andersen}, {\em The mosek interior point
  optimizer for linear programming: an implementation of the homogeneous
  algorithm}, in High performance optimization, Springer, 2000, pp.~197--232.

\bibitem{bauschke2011convex}
{\sc H.~H. Bauschke, P.~L. Combettes, et~al.}, {\em Convex analysis and
  monotone operator theory in Hilbert spaces}, vol.~408, Springer, 2011.

\bibitem{borwein2006convex}
{\sc J.~Borwein and A.~Lewis}, {\em Convex Analysis}, Springer, 2006.

\bibitem{boyd2004convex}
{\sc S.~P. Boyd and L.~Vandenberghe}, {\em Convex optimization}, Cambridge
  university press, 2004.

\bibitem{chen1999proximal}
{\sc X.~Chen and M.~Fukushima}, {\em Proximal quasi-newton methods for
  nondifferentiable convex optimization}, Mathematical Programming, 85 (1999),
  pp.~313--334.

\bibitem{clarke2013functional}
{\sc F.~Clarke}, {\em Functional analysis, calculus of variations and optimal
  control}, vol.~264, Springer, 2013.

\bibitem{dennis1996numerical}
{\sc J.~E. Dennis~Jr and R.~B. Schnabel}, {\em Numerical methods for
  unconstrained optimization and nonlinear equations}, SIAM, 1996.

\bibitem{domahidi2013ecos}
{\sc A.~Domahidi, E.~Chu, and S.~Boyd}, {\em Ecos: An socp solver for embedded
  systems}, in 2013 European control conference (ECC), IEEE, 2013,
  pp.~3071--3076.

\bibitem{elango2022customized}
{\sc P.~Elango, A.~G. Kamath, Y.~Yu, B.~Acikmese, M.~Mesbahi, and J.~M.
  Carson}, {\em A customized first-order solver for real-time powered-descent
  guidance}, in AIAA SciTech 2022 Forum, 2022, p.~0951.

\bibitem{elango2024successive}
{\sc P.~Elango, D.~Luo, S.~Uzun, T.~Kim, and B.~Acikmese}, {\em Successive
  convexification for trajectory optimization with continuous-time constraint
  satisfaction}, arXiv preprint arXiv:2404.16826,  (2024).

\bibitem{cvx}
{\sc M.~Grant and S.~Boyd}, {\em {CVX}: Matlab software for disciplined convex
  programming, version 2.1}.
\newblock \url{https://cvxr.com/cvx}, Mar. 2014.

\bibitem{gurobi}
{\sc {Gurobi Optimization, LLC}}, {\em {Gurobi Optimizer Reference Manual}},
  2024, \url{https://www.gurobi.com}.

\bibitem{hare2004identifying}
{\sc W.~L. Hare and A.~S. Lewis}, {\em Identifying active constraints via
  partial smoothness and prox-regularity}, Journal of Convex Analysis, 11
  (2004), pp.~251--266.

\bibitem{houska2011auto}
{\sc B.~Houska, H.~J. Ferreau, and M.~Diehl}, {\em An auto-generated real-time
  iteration algorithm for nonlinear mpc in the microsecond range}, Automatica,
  47 (2011), pp.~2279--2285.

\bibitem{jerez2014embedded}
{\sc J.~L. Jerez, P.~J. Goulart, S.~Richter, G.~A. Constantinides, E.~C.
  Kerrigan, and M.~Morari}, {\em Embedded online optimization for model
  predictive control at megahertz rates}, IEEE Transactions on Automatic
  Control, 59 (2014), pp.~3238--3251.

\bibitem{jiang2023bregman}
{\sc X.~Jiang and L.~Vandenberghe}, {\em Bregman three-operator splitting
  methods}, Journal of Optimization Theory and Applications, 196 (2023),
  pp.~936--972.

\bibitem{jones2012fast}
{\sc C.~N. Jones, A.~Domahidi, M.~Morari, S.~Richter, F.~Ullmann, and
  M.~Zeilinger}, {\em Fast predictive control: Real-time computation and
  certification}, IFAC Proceedings Volumes, 45 (2012), pp.~94--98.

\bibitem{kamath2023customized}
{\sc A.~G. Kamath, P.~Elango, T.~Kim, S.~Mceowen, Y.~Yu, J.~M. Carson,
  M.~Mesbahi, and B.~Acikmese}, {\em Customized real-time first-order methods
  for onboard dual quaternion-based 6-dof powered-descent guidance}, in AIAA
  SciTech 2023 Forum, 2023, p.~2003.

\bibitem{lewis2002active}
{\sc A.~S. Lewis}, {\em Active sets, nonsmoothness, and sensitivity}, SIAM
  Journal on Optimization, 13 (2002), pp.~702--725.

\bibitem{Lofberg2004}
{\sc J.~L{\"{o}}fberg}, {\em Yalmip : A toolbox for modeling and optimization
  in matlab}, in In Proceedings of the CACSD Conference, Taipei, Taiwan, 2004.

\bibitem{malyuta2022convex}
{\sc D.~Malyuta, T.~P. Reynolds, M.~Szmuk, T.~Lew, R.~Bonalli, M.~Pavone, and
  B.~A{\c{c}}{\i}kme{\c{s}}e}, {\em Convex optimization for trajectory
  generation: A tutorial on generating dynamically feasible trajectories
  reliably and efficiently}, IEEE Control Systems Magazine, 42 (2022),
  pp.~40--113.

\bibitem{mordukhovich2014full}
{\sc B.~S. Mordukhovich, J.~V. Outrata, and M.~E. Sarabi}, {\em Full stability
  of locally optimal solutions in second-order cone programs}, SIAM Journal on
  Optimization, 24 (2014), pp.~1581--1613.

\bibitem{neunert2016fast}
{\sc M.~Neunert, C.~De~Crousaz, F.~Furrer, M.~Kamel, F.~Farshidian,
  R.~Siegwart, and J.~Buchli}, {\em Fast nonlinear model predictive control for
  unified trajectory optimization and tracking}, in 2016 IEEE international
  conference on robotics and automation (ICRA), IEEE, 2016, pp.~1398--1404.

\bibitem{nocedal1999numerical}
{\sc J.~Nocedal and S.~J. Wright}, {\em Numerical optimization}, Springer,
  1999.

\bibitem{odonoghue:21}
{\sc B.~O'Donoghue}, {\em Operator splitting for a homogeneous embedding of the
  linear complementarity problem}, {SIAM} Journal on Optimization, 31 (2021),
  pp.~1999--2023.

\bibitem{rockafellar1970convex}
{\sc R.~T. Rockafellar}, {\em Convex Analysis}, vol.~28 of Princeton
  Mathematical Series, Princeton University Press, Princeton, NJ, 1970,
  \url{https://doi.org/10.1515/9781400873173}.

\bibitem{saad1986gmres}
{\sc Y.~Saad and M.~H. Schultz}, {\em Gmres: A generalized minimal residual
  algorithm for solving nonsymmetric linear systems}, SIAM Journal on
  scientific and statistical computing, 7 (1986), pp.~856--869.

\bibitem{sopasakis2019superscs}
{\sc P.~Sopasakis, K.~Menounou, and P.~Patrinos}, {\em Superscs: fast and
  accurate large-scale conic optimization}, in 2019 18th European Control
  Conference (ECC), IEEE, 2019, pp.~1500--1505.

\bibitem{stellato2020osqp}
{\sc B.~Stellato, G.~Banjac, P.~Goulart, A.~Bemporad, and S.~Boyd}, {\em Osqp:
  An operator splitting solver for quadratic programs}, Mathematical
  Programming Computation, 12 (2020), pp.~637--672.

\bibitem{szmuk2020successive}
{\sc M.~Szmuk, T.~P. Reynolds, and B.~A{\c{c}}{\i}kme{\c{s}}e}, {\em Successive
  convexification for real-time six-degree-of-freedom powered descent guidance
  with state-triggered constraints}, Journal of Guidance, Control, and
  Dynamics, 43 (2020), pp.~1399--1413.

\bibitem{themelis2019supermann}
{\sc A.~Themelis and P.~Patrinos}, {\em Supermann: a superlinearly convergent
  algorithm for finding fixed points of nonexpansive operators}, IEEE
  Transactions on Automatic Control, 64 (2019), pp.~4875--4890.

\bibitem{wang2009fast}
{\sc Y.~Wang and S.~Boyd}, {\em Fast model predictive control using online
  optimization}, IEEE Transactions on control systems technology, 18 (2009),
  pp.~267--278.

\bibitem{wright1993interior}
{\sc S.~J. Wright}, {\em Interior point methods for optimal control of discrete
  time systems}, Journal of Optimization Theory and Applications, 77 (1993),
  pp.~161--187.

\bibitem{yu2020proportional}
{\sc Y.~Yu, P.~Elango, and B.~A{\c{c}}{\i}kme{\c{s}}e}, {\em
  Proportional-integral projected gradient method for model predictive
  control}, IEEE Control Systems Letters, 5 (2020), pp.~2174--2179.

\bibitem{yu2022extrapolated}
{\sc Y.~Yu, P.~Elango, B.~A{\c{c}}{\i}kme{\c{s}}e, and U.~Topcu}, {\em
  Extrapolated proportional-integral projected gradient method for conic
  optimization}, IEEE Control Systems Letters, 7 (2022), pp.~73--78.

\bibitem{zhang2020globally}
{\sc J.~Zhang, B.~O'Donoghue, and S.~Boyd}, {\em Globally convergent type-i
  anderson acceleration for nonsmooth fixed-point iterations}, SIAM Journal on
  Optimization, 30 (2020), pp.~3170--3197.

\bibitem{zometa2013muao}
{\sc P.~Zometa, M.~K{\"o}gel, and R.~Findeisen}, {\em $\mu$ao-mpc: A free code
  generation tool for embedded real-time linear model predictive control}, in
  2013 American Control Conference, IEEE, 2013, pp.~5320--5325.

\end{thebibliography}
\end{document}